\documentclass[fleqn,reqno,a4paper]{amsart}
\usepackage[italian,english]{babel}
\usepackage{eucal,amsfonts,amssymb,amsmath,amsthm,epsfig,mathrsfs}
\usepackage{dsfont}
\usepackage{color}

\usepackage{amscd,amsxtra}
\usepackage{enumerate}
\usepackage{latexsym}

\allowdisplaybreaks

\makeatletter \@addtoreset{equation}{section}

\makeatother \makeatletter

\newtheorem{thm}{Theorem}[section]
\newtheorem{hyp}[thm]{Hypotheses}{\rm}
{\rm}

\newtheorem{coro}[thm]{Corollary}
\newtheorem{prop}[thm]{Proposition}
\newtheorem{defi}[thm]{Definition}
\newtheorem{rmk}[thm]{Remark}{\rm}
\newtheorem{example}[thm]{Example}

\newcommand{\R}{{\mathbb R}}
\newcommand{\N}{{\mathbb N}}

\newcommand{\Rd}{\mathbb R^d}
\newcommand{\Rm}{\mathbb R^m}

\newcommand{\bd}{\begin{defi}}
\newcommand{\ed}{\end{defi}}

\newcommand{\be}{\begin{equation}}
\newcommand{\ee}{\end{equation}}
\newcommand{\barr}{\begin{array}}
\newcommand{\earr}{\end{array}}
\newcommand{\bmn}{\begin{eqnarray}}
\newcommand{\emn}{\end{eqnarray}}
\newcommand{\bln}{\begin{subequations}}
\newcommand{\eln}{\end{subequations}}

\newcommand{\ba}{\begin{align}}
\newcommand{\ea}{\end{align}}
\newcommand{\banm}{\begin{align*}}
\newcommand{\eanm}{\end{align*}}
\newcommand{\one}{\mbox{$1\!\!\!\;\mathrm{l}$}}

\newcommand{\A}{\mathcal A}
\newcommand{\vv}{{\bf v}}

\newcommand{\f}{{\bf f}}

\newcommand{\g}{{\bf g}}
\newcommand{\h}{{\bf h}}

\newcommand{\G}{{\bf G}}

\newcommand{\uu}{{\bf u}}

\allowdisplaybreaks

\begin{document}

\title[$L^p$ estimates for parabolic systems]{$L^p$-estimates for parabolic systems with unbounded coefficients
coupled at zero and first order}
\author[L. Angiuli, L. Lorenzi, D. Pallara]{Luciana Angiuli, Luca Lorenzi and Diego Pallara}
\address{L.A., D.P.: Dipartimento di Matematica e Fisica ``Ennio De Giorgi'', Universit\`a del Salento, Via per Arnesano, I-73100 LECCE (Italy)}
\address{L.L.: Dipartimento di Matematica e Informatica, Universit\`a degli Studi di Parma, Parco Area delle Scienze 53/A, I-43124 PARMA (Italy)}
\email{luciana.angiuli@unisalento.it}
\email{luca.lorenzi@unipr.it}
\email{diego.pallara@unisalento.it}

\date{}


\begin{abstract}
We consider a class of nonautonomous parabolic first-order coupled systems in the Lebesgue space $L^p(\Rd; \Rm)$,
$(d,m \ge 1)$ with $p\in [1,+\infty)$. Sufficient conditions for the associated evolution operator
${\bf G}(t,s)$ in $C_b(\Rd;\R^m)$ to extend to a strongly continuous operator in $L^p(\Rd; \Rm)$ are given.
Some $L^p$-$L^q$ estimates are also established together with $L^p$ gradient estimates.
\end{abstract}

\maketitle

\section{Introduction}
Second order elliptic and parabolic operators with unbounded coefficients have received a great deal of attention because
of their analytical interest as well as their applications to stochastic analysis, both in the autonomous and, more recently,
in the nonautonomous case. Due to the applications in Stochastics, much of the work has been
done in spaces of continuous and bounded functions and in the $L^p$-spaces with respect to {\em the invariant measure}, in the
autonomous, and {\em evolution systems of measures}, in the nonautonomous case. The
existence of a unique classical solution for homogeneous parabolic Cauchy problems associated with
operators with unbounded coefficients in spaces of continuous and bounded functions, or equivalently the existence
of a {\em semigroup} $T(t)$ or an {\em evolution operator} $G(t,s)$, respectively, can be shown under mild assumptions
on the growth of the coefficients. Let us refer the reader to \cite{MetPalWac1,BerLorbook,LorSurvey} and their bibliographies
for more information.

On the other hand, the analysis in the $L^p$ setting with respect to the Lebesgue measure has an independent
analytical interest  and it turns out to be much more difficult than the analysis in the space
of continuous and bounded functions or in $L^p$-spaces with respect to the
invariant measure (resp. evolution system of measures). Even in the autonomous case, the Cauchy problem may be not well posed
in $L^p(\Rd, dx)$ if the coefficients are unbounded, unless they satisfy very restrictive assumptions.
For instance, in the $1$-dimensional case very simple operators, such as
$D^2-|x|^\varepsilon x D$, with $\varepsilon>0$, do not generate any semigroup in $L^p(\R, dx)$ and in this
situation, the lack of the potential term plays a crucial role, see also \cite{AreMetPal06Scr} for further examples and comments.

Since nowadays many of the results obtained concern the single equations, the aim of this paper is the
study of parabolic systems with unbounded coefficients, coupled in the zero and first order terms, in the Lebesgue
space $L^p(\Rd,\R^m)$. We consider the Cauchy problem
\begin{equation}
\left\{
\begin{array}{lll}
D_t\uu(t,x)=(\boldsymbol{\mathcal A}(t)\uu)(t,x),\quad\quad & t>s\in I, &x\in\R^d,\\[1mm]
\uu(s,x)=\f(x), &&x\in\R^d
\end{array}
\right.
\label{eq:cauchy_problem_system}
\end{equation}
where $I$ is an open right-halfline or the whole $\R$ and the elliptic operators
\begin{equation}
\boldsymbol{\mathcal A}{\bf v}=\sum_{i,j=1}^dD_i(q_{ij}D_{j}{\bf v})+\sum_{i=1}^d B_iD_i{\bf v}+C{\bf v}
\label{operat-A}
\end{equation}
have unbounded coefficients  $q_{ij}:I\times\Rd\to \R$ and $B_i, C: I\times\Rd \to \R^{m^2}$ ($m \ge 1$).

Second order elliptic and parabolic systems have been already studied in the simplest case of {\em zero order coupling}, i.e., when $B_i=b_iI_m$
(see \cite{HieLorPruRhaSch09Glo,DelLor11OnA}).
The more general frame of {\em first order coupling}, i.e., uncoupled diffusion and coupled drift and potential, has been very recently
studied in the space of continuous and bounded functions in \cite{AddAngLor15Cou},
where the existence of an {\em evolution operator} ${\bf G}(t,s)$ associated with
$\boldsymbol{\mathcal A}(t)$ in $C_b(\Rd; \Rm)$ has been shown. Here, we take advantage of such construction and
of a pointwise estimate shown in \cite{AddAngLor15Cou} to start our investigation on the properties of
 ${\bf G}(t,s)$ in the $L^p$ context. We refer to \cite{LunardiBook,CL} for the abstract theory of evolution operators.

We assume that the coefficients are regular enough, namely locally $C^{\alpha/2, \alpha}$, for some $\alpha \in (0,1)$,
together with the first order spatial derivatives of $q_{ij}$ and of the entries of $B_i$, for any $i,j=1, \ldots, d$, and that
the matrix $Q(t,x)=[q_{ij}(t,x)]_{i,j=1, \ldots,d}$ is uniformly positive definite, see Hypotheses \ref{hyp_base}.

The $L^p$ analysis is carried out under two different sets of assumptions, Hypotheses \ref{uni1} and \ref{uni2}, which we compare in
Remark \ref{hypocomparison}. The two approaches give slightly different results.
Indeed, under Hypotheses \ref{uni1} we deal directly with the vectorial problem. Using the pointwise estimate
proved in \cite{AddAngLor15Cou} (and recalled in the Appendix),
an interpolation argument and requiring
a balance between the growth of the potential matrix $C$ and the derivative of the drift matrices $B_i$ ($i=1, \ldots,d$),
we prove that the evolution operator ${\bf G}(t,s)$ extends to a bounded and strongly continuous
operator in $L^p(\Rd; \Rm)$ for any $p \in [1,+\infty)$.

On the other hand, when Hypotheses \ref{uni2} are satisfied, we estimate
$|{\bf G}(t,s)\f|^p$ in terms of $G(t,s)|\f|^p$ for any $t>s \in I$, $p \in [p_0, +\infty)$ and some $p_0>1$.
Here, $G(t,s)$ is the evolution operator which governs an auxiliary scalar problem.  As a consequence of this comparison result,
the boundedness of ${\bf G}(t,s)$ in $\mathcal{L}(L^p(\Rd;\Rm))$ for $p\in [p_0,+\infty)$ can be obtained as a byproduct
of the boundedness of $G(t,s)$ in $\mathcal{L}(L^1(\Rd))$. Sufficient conditions in order that $G(t,s)$ is bounded in $L^p$
for any $p \in [1,+\infty)$ can be found in \cite{AngLor10Com}. Notice however that slightly strengthening Hypothesis \ref{uni2}(ii)
we can deal with the whole scale of $1<p<\infty$ rather than $p\geq p_0$, see Remark \ref{alpha}.

Going further, we find conditions for the hypercontractivity of ${\bf G}(t,s)$.
More precisely, under suitable assumptions, we prove that
\begin{equation}\label{iper_intro}
\|{\bf G}(t,s)\f\|_{L^q(\Rd;\Rm)}\le c\|\f\|_{L^p(\Rd;\Rm)},
\end{equation}
for any $t\in (s, T]$, $T>s \in I$, $\f \in L^p(\Rd;\Rm)$, $q\geq p$ and some positive constant $c$ depending on $p,q,s$ and $T$.
Actually, whenever Hypotheses \ref{uni1} are satisfied,
under the same assumptions which guarantee that $L^p(\Rd,\Rm)$ is preserved by the action of ${\bf G}(t,s)$, we prove \eqref{iper_intro}
for any $2 \le p \le q$. Then, arguing by duality we establish \eqref{iper_intro} also when $1 \le p \le q \le 2$.
Applying this hypercontractivity result to the scalar evolution operator $G(t,s)$ and using
the pointwise estimate of $|{\bf G}(t,s)\f|^p$ in terms of $G(t,s)|\f|^p$, we provide conditions for
\eqref{iper_intro} to hold for $p_0 \le p \le q$, when Hypotheses \ref{uni2} are satisfied.

The hypercontractivity estimate \eqref{iper_intro}, in this generality,  seems to be new also in the autonomous
scalar case. Some $L^p$-$L^q$
estimates have been recently proved in \cite{IoMeSoSp14LpLq} for a special class of homogeneous operators with unbounded diffusion.

Next, we prove some pointwise estimates for the spatial derivatives of ${\bf G}(t,s)\f$.
Under additional assumptions, which are essentially growth conditions on the coefficients of the operator
$\boldsymbol{\mathcal A}(t)$ and their derivatives, we show that there exist positive constants $c_1, c_2$ such that
\begin{equation}\label{stima_grad1_intro}
|D_x {\bf G}(t,s)\f|^p \le c_1 G(t,s)(|\f|^p+|D\f|^p)
\end{equation}
and, under more restrictive conditions, that
\begin{equation}\label{lp-w1p-intro}
|D_x {\bf G}(t,s)\f|^p\le c_2 (t-s)^{-\frac{p}{2}}G(t,s)|\f|^p,
\end{equation}
for any $t\in (s, T]$, $T>s \in I$, $\f \in C^1_c(\Rd;\Rm)$ and $p \in [p_1,+\infty) $ for some $p_1>1$.

Now, if the scalar evolution operator $G(t,s)$ preserves $L^1(\Rd)$, estimates \eqref{stima_grad1_intro}
and \eqref{lp-w1p-intro} yield that the evolution operator ${\bf G}(t,s)$ belongs to
$\mathcal{L}(W^{1,p}(\Rd;\Rm))$ and to $\mathcal{L}(L^p(\Rd;\Rm),W^{1,p}(\Rd;\Rm))$, respectively.
As a consequence of this fact, we show that ${\bf G}(t,s)$ is bounded from $W^{\theta_1,p}(\Rd; \Rm)$ into
$W^{\theta_2,p}(\Rd;\Rm)$ for any $0 \le \theta_1 \le \theta_2 \le 1$ and any $p \ge p_1$.

We believe that estimates \eqref{stima_grad1_intro} and \eqref{lp-w1p-intro} could represent a helpful tool to study
the evolution operator ${\bf G}(t,s)$ in $L^p$-spaces with respect to {\em a natural extension to the vector case of
evolution systems of measures}, whose definition and analysis is deferred to a future paper.
Indeed, already in the scalar case, (see \cite{AngLorLun12Asy, AngLor12OnI}), pointwise gradient estimates
have been a key tool to study the asymptotic behaviour of
the evolution operator associated with the problem and in establishing some summability improving results for such operator
in the $L^p$ spaces with respect the tight time dependent family of invariant measures.

The last section of the paper is devoted to exhibit some classes of operators which satisfy our assumptions.

\subsection*{Acknowledgements} The authors have been partially supported by the PRIN 2010 MIUR project ``Problemi differenziali di evoluzione: 
approcci deterministici e stocastici e loro interazioni" and are members of G.N.A.M.P.A. of the Italian Istituto Nazionale di Alta Matematica 
(INdAM). L.A. and L.L. have also been supported by ``Progetto GNAMPA 2014: Equazioni ellittiche e paraboliche''. 

\subsection*{Notations}
Functions with values in $\R^m$ are displayed in bold style. Given a function $\f$ (resp. a sequence
$(\f_n)$) as above, we denote by $f_i$ (resp. $f_{n,i}$) its $i$-th component (resp. the $i$-th component
of the function $\f_n$).
By $B_b(\Rd;\Rm)$ we denote the set of all the bounded Borel measurable functions $\f:\Rd\to\Rm$.
For any $k\ge 0$, $C^k_b(\R^d;\R^m)$ is the space of all the functions
whose components belong to $C^k_b(\R^d)$, where the notation $C^k(\R^d)$ ($k\ge 0$) is standard and we use the subscript ``$c$'' and ``$b$''
for spaces of functions with compact support and bounded, respectively.
Similarly, when $k\in (0,1)$, we use the subscript ``loc'' to denote the space of all $f\in C(\Rd)$
which are H\"older continuous in any compact set of $\Rd$.
We assume that the reader is familiar also with the parabolic spaces $C^{\alpha/2,\alpha}(I\times \Rd)$
($\alpha\in (0,1)$) and $C^{1,2}(I\times \Rd)$, and we use the subscript ``loc'' with the same meaning as above.

The Euclidean inner product of the vectors $x,y\in\R^d$ is denoted by $\langle x,y\rangle$.
For any square matrix $M$, we denote by $M_{ij}$, ${\rm Tr}(M)$ and $M^*$ the $ij$-th
element of the matrix $M$, the trace of $M$ and the matrix transposed to $M$, respectively.
Finally, $\lambda_M$ and $\Lambda_M$ denote the minimum and the maximum eigenvalue of
the (symmetric) matrix $M$. For any $k \in \N$, by $I_k$ we denote the identity matrix of size $k$.
Square matrices of size $m$ are thought as elements of $\R^{m^2}$.

By $\chi_A$, $\one$ and ${\bf e}_j$ we denote the characteristic function of the set
$A\subset\R^d$, the function which is identically equal to 1 in $\R^d$ and the $j$-th vector of the Euclidean basis of $\R^m$.
Finally, the Euclidean open ball with centre $x_0$ and radius $R>0$ and its closure are
denoted by $B_{R}(x_0)$ and $\overline B_R(x_0)$; when $x_0=0$ we simply
 write $B_R$ and $\overline B_R$.

For any interval $J\subset \R$ we denote by $\Sigma_J$ the set $\{(t,s)\in J\times J:\;\, t>s\}$.

\section{Preliminary results}

Let $I$ be an open right-halfline (possibly $I=\R$) and $\{\boldsymbol{\mathcal A}(t)\}_{t \in I}$
be the family  of second order uniformly elliptic operators defined in \eqref{operat-A}.
In this paper we study the Cauchy problem  \eqref{eq:cauchy_problem_system}
when $\f \in L^p(\Rd;\Rm)$ and $s\in I$, under the following standing assumptions.

\noindent
\begin{hyp}
\label{hyp_base}
\begin{enumerate}[\rm (i)]
\item
The matrices $Q=[q_{ij}]_{i,j=1, \ldots,d}$, $B_i$ $(i=1,\ldots,d)$ and $C$ are symmetric. Further,
$q_{ij}, (B_i)_{lk}\in C^{\alpha/2,1+\alpha}_{\rm loc}(I\times \Rd)$ and $C_{lk}\in C^{\alpha/2,\alpha}_{\rm loc}(I\times\R^d)$
for any $i,j=1,\ldots,d$ and $l,k =1,\ldots,m$;
\item
the matrix $Q$ is uniformly elliptic, i.e., $\nu_0:=\inf_{I\times \Rd}\lambda_Q (t,x)>0$ where
\begin{equation*}
\lambda_Q(t,x):=\min\{\langle Q(t,x)\xi,\xi\rangle:\, \xi \in \Rd,\ |\xi|=1\},\qquad\;\,t\in I, \;\, x\in\R^d
\end{equation*}
is the minimum eigenvalue of $Q(t,x)$.
\end{enumerate}
\end{hyp}
Besides Hypotheses \ref{hyp_base} we consider one of the following two sets of assumptions.

\noindent
\begin{hyp}\label{uni1}
\begin{enumerate}[\rm(i)]
\item
The function ${\mathcal K}_\eta: I\times\R^d\to \R$, defined by
\begin{equation}
{\mathcal K}_\eta=\sum_{i,j=1}^d(Q^{-1})_{ij}[\langle B_i\eta,\eta\rangle\langle
B_j\eta,\eta\rangle-\langle B_i\eta,B_j\eta\rangle]-4\langle C\eta,\eta\rangle,
\label{eq:positive condition}
\end{equation}
is nonnegative in $I \times \Rd$, for any $\eta \in \partial B_1$;
\item
for any bounded interval $J\subset I$ there exist a constant $\lambda_J$ and a positive ({\em Lyapunov}) function
 $\varphi_J\in C^2(\R^d)$, blowing up as $|x|\to +\infty$, such that
\begin{equation*}
\sup_{\eta \in \partial B_1}\sup_{(t,x)\in J\times\R^d}(\A_{\eta}(t)\varphi_J)(x)-\lambda_J\varphi_J(x))<+\infty,
\end{equation*}
where
\begin{equation}\label{defAeta}
\A_{\eta}={\rm div}(Q D_x)+\langle b_{\eta},D_x\rangle,\qquad
(b_{\eta})_i=\langle B_i\eta,\eta\rangle.
\end{equation}
\end{enumerate}
\end{hyp}
Condition \ref{uni1}(i) is already used by \cite{KreMaz12Max} in the case of bounded coefficients.

\noindent
\begin{hyp}\label{uni2}
\begin{enumerate}[\rm (i)]
\item
There exist functions $b_i:I\times \Rd \to \R$ and $\tilde{B}_i:I \times \R^d\to \R^{m^2}$ such that
$B_i:=b_iI_m+\tilde B_i$, for any $i=1, \ldots,d$, $\sigma>0$, and a function $\xi:I \to (0,+\infty)$ such that
\[
|(\tilde{B}_{i})_{jk}(t,x)|\leq \xi(t)\lambda_Q^{\sigma}(t,x), \qquad\;\, (t,x)\in I\times \Rd,
\]
for any $j,k=1,\ldots,m$ and $i=1,\ldots,d$;
\item
for any bounded interval $J\subset I$ there exists $\beta \geq 1/4$ such that
\begin{equation}\label{HpJ}
H_{\beta,J}:=\sup_{J\times\Rd}(\Lambda_C+\beta dm^2\xi^2\lambda_Q^{2\sigma-1})<+\infty;
\end{equation}
\item
for any bounded interval $J\subset I$ there exist $\lambda_J>0$ and a
positive function $\varphi_J\in C^2(\R^d)$ blowing up as $|x|\to +\infty$ such
that $\sup_{J\times\R^d}(\A\varphi_J-\lambda_J\varphi_J)<+\infty$,
where
\begin{equation}\label{defA}
\A={\rm div}(Q D_x)+\langle b,D_x\rangle,\qquad b=(b_1,\ldots,b_m).
\end{equation}
\end{enumerate}
\end{hyp}

\begin{rmk}\label{luftansa}{\rm Hypothesis $\ref{uni1}$(i)
can be replaced with the weaker condition
\begin{equation}\label{weak}
\inf_{\eta\in \partial B_1}\inf_{J \times \Rd}\mathcal K_\eta >-\infty
\end{equation}
for any bounded interval $J \subset I$.
Indeed, in this latter case, for any bounded interval $J \subset I$ there exists a positive constant $c_J$ such that
$\mathcal K_\eta\geq -c_J$  in $J\times\R^d$ for any $\eta \in\partial B_1$. Let us notice that
 $\uu$ is a classical solution of the Cauchy problem \eqref{eq:cauchy_problem_system} if and only if the function $\vv$,
defined by $\vv(t,x):=e^{-c_J(t-s)/4}\uu(t,x)$ for any $(t,x)\in (s,+\infty)\times \Rd$, is a classical solution of the problem
\begin{equation}\label{weak_pb}
\left\{
\begin{array}{ll}
D_t\vv(t,x)=\displaystyle\left (\boldsymbol{\mathcal A}(t)-\frac{c_J}{4}\right )\vv(t,x), \qquad\,&(t,x)\in (s,+\infty)\times \Rd \\[2mm]
\vv(s,x)=\f(x), \qquad &x \in \Rd.
\end{array}
\right.
\end{equation}
The elliptic operator in problem \eqref{weak_pb} satisfies Hypothesis $\ref{uni1}$(i) and, clearly, the uniqueness of $\vv$ is equivalent
to the uniqueness of $\uu$.}
\end{rmk}

\begin{rmk}\label{hypocomparison} {\rm
 A comparison between Hypotheses \ref{uni1} and \ref{uni2} is in order. First of all, notice that writing the
 matrices $B_i$ as in \ref{uni2}(i) the function ${\mathcal K}_\eta$ depends only upon $\tilde{B}_i$, because the
 diagonal part cancels. The two sets of hypotheses are independent in
 general: \ref{uni2}(i) and (ii) imply \ref{uni1}(i), whereas \ref{uni1}(ii) is stronger
 than \ref{uni2}(iii). Indeed, assuming \ref{uni2}(i) it is easily seen that
 \[
  \sum_{i,j=1}^d(Q^{-1})_{ij}[\langle B_i\eta,\eta\rangle\langle B_j\eta,\eta\rangle-\langle B_i\eta,B_j\eta\rangle]
 \]
is negative and of order $\lambda_Q^{2\sigma-1}$. This fact together with \ref{uni2}(ii) implies \ref{uni1}(i) (taking
Remark \ref{luftansa} into account). On the other hand, assuming \ref{uni2}(i), the function ${\mathcal K}_\eta$
can be of order less than $\lambda_Q^{1-2\sigma}$. For instance, assume $d=m=2$, $Q={\rm diag}(\lambda_Q,\Lambda_Q)$,
$B_1=b_1I_2$ diagonal and $\tilde{B}_2\neq 0$. Then, we have
\[
 {\mathcal K}_\eta = \Lambda_Q^{-1}(\langle\tilde{B}_2\eta,\eta\rangle^2-|\tilde{B}_2\eta|^2)-4\langle C\eta,\eta\rangle
 \geq 0 \quad {\rm if} \quad \Lambda_C+2\xi^2 \lambda_Q^{2\sigma}\Lambda_Q^{-1}<+\infty ,
\]
which is weaker than \eqref{HpJ} if $\lambda_Q=o(\Lambda_Q)$.

Concerning \ref{uni1}(ii) and \ref{uni2}(iii), the latter requires the existence of a Lyapunov function for
{\em one} decomposition of each drift matrix, while the former requires the existence of a Lyapunov function for {\em any}
decomposition $B_i = b_\eta I_m + \tilde{B}_{\eta,i}$, $\eta \in \partial B_1$.
}\end{rmk}

We start by recalling some known results used in the sequel and proved in \cite{AddAngLor15Cou}. The evolution operator on $C_b(\Rd;\Rm)$
which gives a solution of problem \eqref{eq:cauchy_problem_system}
is obtained as the limit of the sequence of the evolution operators related to the
following Cauchy-Dirichlet problem in $I\times B_n$:
\begin{equation}\label{prob_approx}
\left\{
\begin{array}{lll}
D_t\uu_n(t,x)=(\boldsymbol{\mathcal A}(t)\uu_n)(t,x), \qquad\quad& t>s, \,x\in B_n,\\
\uu_n(t,x)={\bf 0}, & t>s,\,x\in\partial B_n,\\
\uu_n(s,x)=\f(x), & x\in\overline{B_n}.
\end{array}
\right.
\end{equation}
We shall also be concerned with the Cauchy-Neumann problem in $I\times B_n$:
\begin{equation}\label{prob_approx_Neu}
\left\{
\begin{array}{lll}
D_t\uu_n(t,x)=(\boldsymbol{\mathcal A}(t)\uu_n)(t,x), \qquad\quad& t>s, \,x\in B_n,\\[1mm]
\displaystyle\frac{\partial\uu_n}{\partial\nu}(t,x)={ \bf 0}, & t>s,\,x\in\partial B_n,\\[1mm]
\uu_n(s,x)=\f(x), & x\in\overline{B_n},
\end{array}
\right.
\end{equation}
where $\nu$ denotes the unit exterior normal vector to $\partial B_n$.
Throughout the paper, we denote by $\G_n^{\mathcal D}(t,s)$ and $\G_n^{\mathcal N}(t,s)$ the
Dirichlet and Neumann evolution operators associated with problems \eqref{prob_approx}, \eqref{prob_approx_Neu} in $C_b(B_n;\Rm)$.

\begin{prop}\label{recall}
Under Hypotheses $\ref{hyp_base}$ and $\ref{uni1}$ $($resp. $\ref{uni2})$, for any $\f \in C_b(\Rd;\Rm)$, problem
\eqref{eq:cauchy_problem_system}
admits a unique classical solution $\uu$ which is bounded in the strip $[s,T]\times\R^d$ for any $T>s \in I$.
Setting ${\bf G}(t,s)\f:= \uu(t, \cdot)$ for any $t>s$ and $\f \in C_b(\Rd;\Rm)$, 
 $\G(t,s)$ is a bounded linear operator in $C_b(\Rd;\Rm)$ and
\begin{equation}
\|\G(t,s)\f\|_{\infty}\leq \gamma(t-s)\|\f\|_{\infty}, \qquad\;\,t\in (s,T),
\label{eq:norma_condition_classical solution}
\end{equation}
where $\gamma(r)=1$ $($resp.\footnote{Here $H_{1/4,[s,T]}$ is the constant in \eqref{HpJ}.} $\gamma(r)=e^{H_{1/4,[s,T]}r})$ 
for any $r>0$. Moreover, for any $s \in I$ and $\f\in C_b(\Rd;\Rm)$, both $\G_n^{\mathcal N}(\cdot,s)\f$ and 
$\G_n^{\mathcal D}(\cdot,s)\f$ converge to $\G(\cdot,s)\f$ in $C_{\rm loc}^{1,2}((s,+\infty)\times \Rd;\R^m)$.
\end{prop}

The uniqueness of the solution of the problem \eqref{eq:cauchy_problem_system} shows
that the family $\{{\bf G}(t,s)\}_{t\geq s\in I}$ is an evolution operator in $C_b(\Rd;\R^m)$.

\begin{rmk}\label{alpha}{\rm
Notice that working in $L^p$ is allowed provided that Hypothesis \ref{uni2}(ii) holds for some $\beta\geq [4(p-1)]^{-1}$, as we shall see in
the proof of Proposition \ref{stimapuntLp} below. We are supposing $\beta\geq 1/4$ in order to encompass the case $p=2$: indeed,  
estimate \eqref{eq:norma_condition_classical solution} has been obtained  as consequence of a
pointwise estimate for $|\uu|^2$ in terms of the solution of a suitable scalar problem. 

Moreover, we point out that if \eqref{HpJ} holds with $\lambda_Q^\alpha$ in place of $\lambda_Q^{2\sigma -1}$
for some $\alpha<2\sigma-1$, then every $\beta>0$ is allowed and we can extend our results to the whole scale
of $p>1$. We shall not mention this extension anymore.
}\end{rmk}

Since in this paper we are interested in studying the evolution operator $\G(t,s)$ in the $L^p(\Rd;\Rm)$ setting
under Hypotheses $\ref{uni2}$, we extend the just mentioned pointwise estimate to $|\uu|^p$ for any $p\in [1+\frac{1}{4\beta},+\infty)$.

\begin{prop}\label{stimapuntLp}
Assume that Hypotheses $\ref{uni2}$ hold true; then, for every bounded interval $J \subset I$ and $p\geq 1+\frac{1}{4\beta}$,
there exists a positive constant $K_{J}$ such that
\begin{equation}\label{pointwise}
|(\G(t,s)\f)(x)|^p\leq e^{pK_{J}(t-s)} (G(t,s)|\f|^p)(x),
\end{equation}
for any $(t,s)\in \Sigma_J$, $x \in \Rd$ and $\f \in C_b(\Rd;\R^m)$, where $G(t,s)$
 denotes the evolution operator in $C_b(\Rd)$ associated with the operator $\mathcal A$ defined in \eqref{defA}. 
 Here, $K_{J}=H_{1/4,J}$ if $p \ge 2$ whereas $K_{J}= H_{\beta,J}$ if $p \in [1+\frac{1}{4\beta},2)$.
\end{prop}
\begin{proof}
Estimate \eqref{pointwise} has been already proved when $p=2$ in \cite[Prop. 2.8]{AddAngLor15Cou} with $K_{J}=H_{1/4,J}$; for a general
$p$, its proof is similar, so that we limit ourselves to sketch it. Moreover, it suffices to prove  \eqref{pointwise}
only for $p\in [1+\frac{1}{4\beta},2)$. Indeed, if $p>2$, the integral representation formula of $G(t,s)|\f|^2$ in terms of the transition
kernels associated with $\mathcal A$ in $C_b(\Rd)$ (see \cite[Prop. 2.4]{KunLorLun09Non}) and the Jensen inequality yield
\[
|\G(t,s)\f|^p\le (e^{2H_{{1/4},J}(t-s)}G(t,s)|\f|^2)^{p/2}\le e^{pH_{1/4,J}(t-s)}G(t,s)|\f|^p
\]
for any $(t,s) \in \Sigma_J$. Hence, \eqref{pointwise} follows.

Now, let $J \subset I$ be a bounded interval. Fix $p\in [1+\frac{1}{4\beta},2]$, $\varepsilon>0$,
 and, for brevity, let $H= H_{\beta,J}$ be as in Hypotheses
$\ref{uni2}$(ii) and $\uu=\G(\cdot,s)\f$.
We set $w_\varepsilon=(|\uu|^2+\varepsilon)^{p/2}$ and
\[
u_\varepsilon(t,\cdot)=e^{-pH(t-s)}w_\varepsilon(t,\cdot)- G(t,s)(|\f|^2+\varepsilon)^{p/2},\qquad\;\, t>s \in I .
\]
The function $u_\varepsilon$ belongs to $C^{1,2}((s,+\infty)\times \Rd)\cap C_b([s,+\infty)\times \Rd)$ and verifies
\begin{align*}
D_t u_\varepsilon- \A u_\varepsilon=p  e^{-pH(t-s)} w_\varepsilon^{1-2/p}
\bigg [&\sum_{i=1}^d\langle\uu,\tilde{B}_i D_i \uu \rangle+\langle\uu, C \uu \rangle
-\sum_{i,j=1}^d q_{ij}\langle D_i\uu,D_j\uu\rangle\\
&+(2-p)(|\uu|^2+\varepsilon)^{-1}\sum_{i,j=1}^d q_{ij}\langle\uu, D_i\uu\rangle \langle \uu, D_j \uu\rangle
-H (|\uu|^2+\varepsilon)\bigg ]
\end{align*}
in $(s,\infty) \times \Rd$. Since
\begin{align}
\sum_{i,j=1}^dq_{ij}\langle\uu,D_i\uu\rangle\langle\uu,D_j\uu\rangle
\le &\sum_{h,k=1}^m|u_h||u_k||\langle QD_xu_h,D_x u_k\rangle|\le \sum_{h,k=1}^m|u_h||u_k||Q^{1/2}D_x u_h||Q^{1/2}D_x u_k|
\nonumber \\ \label{star}
=&\left (\sum_{h=1}^m|u_h||Q^{1/2}D_x u_h|\right )^2 \\ \nonumber
\le & \left (\sum_{h=1}^m|u_h|^2\right )\left (\sum_{h=1}^m|Q^{1/2}D_x u_h|^2\right )=|\uu|^2\sum_{i,j=1}^dq_{ij}\langle D_i\uu,D_j\uu\rangle,
\end{align}
by the assumptions it follows that
\begin{align}\label{form1}
 D_t u_\varepsilon- \A(t) u_\varepsilon\le & p e^{-pH(t-s)}
 w_\varepsilon^{1-\frac{2}{p}}\left[\sum_{i=1}^d\langle\uu,\tilde{B}_i D_i \uu\rangle+(1-p)\lambda_Q |D_x \uu|^2+(\Lambda_C-H)|\uu|^2\right]
\end{align}
in $(s, \infty) \times \Rd$. The Young and the
Cauchy-Schwarz inequalities and Hypotheses $\ref{uni2}$(i) show that
\begin{align}
\sum_{i=1}^d\langle\uu, \tilde{B}_i D_i \uu \rangle+(1-p)\lambda_Q |D_x \uu|^2&\le
m\xi\lambda_Q^{\sigma}|\uu|\sum_{i=1}^d|D_i\uu|+(1-p)\lambda_Q  |D_x \uu|^2\notag\\
&\le (a dm^2\xi^2+1-p)\lambda_Q|D_x\uu|^2+\frac{\lambda_Q^{2\sigma-1}}{4a}|\uu|^2
\label{form2}
\end{align}
in $J \times \Rd$ where and $a=a(t)$ is an arbitrary positive function.
Putting together \eqref{form1}, \eqref{form2}
and choosing $a=(p-1)(dm^2\xi^2)^{-1}$ yield that
\[
D_t u_\varepsilon- \A u_\varepsilon\le p e^{-pH(t-s)} w_\varepsilon^{1-2/p}
\left[\frac{dm^2\xi^2}{4(p-1)}\lambda_Q^{2\sigma-1}+\Lambda_C-H\right]|\uu|^2\le 0
\]
in $((s, \infty)\cap J)\times \Rd$. The maximum principle in \cite[Prop. 2.1]{KunLorLun09Non}
yields that $u_\varepsilon \le 0$ in $((s, \infty)\cap J)\times \Rd$, i.e.,
\[
(|\uu(t,\cdot)|^2+\varepsilon)^{p/2} \le e^{pH(t-s)}G(t,s)(|\f|^2+\varepsilon)^{p/2},\quad\;\, (t,s) \in \Sigma_J.
\]
Letting $\varepsilon \to 0^+$ we get \eqref{pointwise} with $K_{J}= H_{\beta, J}$.
\end{proof}

\section{The evolution operator ${\bf G}(t,s)$ in $L^p(\Rd;\Rm)$}

As it has been already stressed in the introduction, even in the autonomous scalar case, the Cauchy problem \eqref{eq:cauchy_problem_system}
is not well posed in $L^p(\Rd, dx)$ if the coefficients of $\boldsymbol{\mathcal A}$ are unbounded, unless
they satisfy suitable assumptions.

Actually, in some cases the Lebesgue space $L^p(\Rd, dx)$ is not preserved by the action of the evolution operator
associated with $\boldsymbol{\mathcal A}$. For example, the compactness in $C_b(\Rd)$ implies 
that $L^p(\Rd, dx)$ is not preserved (see e.g. \cite{MetPalWac2, AngLor10Com})
by the action of the evolution operator. Here, we are interested in studying properties of the evolution
operators $\G(t,s)$ in $L^p(\Rd;\Rm)$ when this space is preserved by its action and when an estimate like
\begin{equation}\label{est_p}
\|\G(t,s)\f\|_{L^p(\Rd;\Rm)}\le c_p(t-s) \|\f\|_{L^p(\Rd;\Rm)}
\end{equation}
holds true for some function $c_p:[0,+\infty)\to (0,+\infty)$.

In what follows we consider alternatively Hypotheses $\ref{uni1}$ and $\ref{uni2}$, under 
additional assumptions. See also Remark \ref{alpha} in connection to Theorem \ref{th2} and Proposition \ref{def_Lp}. 

We begin by considering the case when Hypotheses \ref{uni1} are satisfied. Here, in order to use a duality
argument we introduce the following conditions.
\noindent
\begin{hyp}\label{uni11}
\begin{enumerate}[\rm(i)]
There exists a function $\kappa:I\times\Rd\to\R$, bounded from above
by a constant $\kappa_0$, such that
\item
the function $\tilde{{\mathcal K}}_\eta:I\times \Rd\to \R$ defined by
\begin{equation*}
\tilde{{\mathcal K}}_\eta={\mathcal K}_\eta+4\sum_{k=1}^d\langle D_k B_k\eta,\eta\rangle+4\kappa ,
\end{equation*}
where ${\mathcal K}_\eta$ is defined in \eqref{eq:positive condition},
is nonnegative in $I \times \Rd$ for any $\eta \in \partial B_1$;
\item
for any bounded interval $J\subset I$ there exist a constant $\lambda_J$ and a
positive function $\varphi_J\in C^2(\R^d)$, blowing
up as $|x|\to +\infty$, such that
\begin{equation*}
\sup_{\eta \in \partial B_1}\sup_{(t,x)\in J\times\R^d}\Bigl((\tilde{\A}_{\eta}(t)
\varphi_J)(x)-\lambda_J\varphi_J(x)\Bigr)<+\infty,
\end{equation*}
where
\begin{equation*}
\tilde{\A}_{\eta}={\rm div}(Q D_x)-\langle b_{\eta},D_x\rangle+2\kappa
\end{equation*}
and $b_\eta$ is defined in \eqref{defAeta}. 
\end{enumerate}
\end{hyp}

\begin{rmk}
{\rm The same arguments as in Remark \ref{luftansa} show that the condition
$\tilde{\mathcal K}_{\eta}\ge 0$ in $J\times\Rd$ can be replaced with the weaker condition
$\inf_{\eta\in \partial B_1}\inf_{J \times \Rd}\tilde{\mathcal K}_{\eta} >-\infty$
for any bounded interval $J \subset I$.
}
\end{rmk}

\begin{thm}\label{th1}
Assume that Hypotheses $\ref{uni1}$ hold true. If for some interval $J\subset I$ there exists a positive
constant $L_J$ such that
\begin{equation}\label{p-2-infty}
\Lambda_{2C- \sum_{i=1}^d D_iB_i}(t,x) \le L_J, \quad\;\,(t,x)\in J \times \Rd,
\end{equation}
then estimate \eqref{est_p} is satisfied for any $(t,s)\in \Sigma_J$,
$\f \in C_c(\Rd;\Rm)$ and $p \in [2,+\infty)$ with $c_p(r)=e^{rL_J/p}$.
In addition, if Hypotheses $\ref{uni11}$ are satisfied, then estimate \eqref{est_p} holds also
for $p \in [1,2)$ with $c_p(r)=e^{r(L_J+ \kappa_0(p'-2))/p'}$, $r \ge 0$ and $p'=p/(p-1)$.
\end{thm}
\begin{proof}
Let us fix $s \in J$, $\f \in C_c(\Rd;\Rm)$ and for any $n\in\N$ consider the classical solution
$\uu_n:=\G_n(\cdot,s)\f=\G_n^{\mathcal D}(\cdot,s)\f$ of the Cauchy-Dirichlet problem \eqref{prob_approx}.
>From Proposition \ref{recall}, $\G_n(\cdot,s)\f$ converges pointwise to $\G(\cdot,s)\f$ as $n \to +\infty$ and
\begin{equation}\label{infty1}
\|\G_n(t,s)\f\|_\infty \le \|\f\|_\infty,\qquad\,\, t \in (s,+\infty).
\end{equation}
Let us prove that estimate \eqref{est_p} holds true for $p=2$ with $\G(t,s)$ replaced by $\G_n(t,s)$ and
some positive function $c$ independent of $n$.
To ease the notation, we use $\|\cdot\|_p$ (resp. $\|\cdot\|_{p,n}$) in place of
$\|\cdot\|_{L^p(\Rd;\Rm)}$ (resp. $\|\cdot\|_{L^p(B_n;\Rm)}$).
To this aim, first observe that from the symmetry of $B_i$ it follows
that $2\langle \vv, B_i D_i \vv\rangle={\rm Tr}(B_i D_i (\vv\otimes \vv))$
for any smooth function $\vv: \Rd\to \Rm$ and $i=1,\ldots,d$. Then,
multiplying the differential equation $D_t \uu_n= \boldsymbol{\mathcal A}(t)\uu_n$
by $\uu_n$ and integrating by parts in $B_n$, we get
\begin{align*}
D_t\|\uu_n(t,\cdot)\|_{2,n}^2 =&\, 2\int_{B_n}\langle \uu_n(t,\cdot),(\boldsymbol{\mathcal A}(t)\uu_n)(t,\cdot) \rangle dx\\
=&-2 \int_{B_n}\langle Q(t,\cdot) D_x\uu_n(t,\cdot),D_x\uu_n(t,\cdot) \rangle dx
-\sum_{i=1}^d\int_{B_n}\langle (D_i B_i)(t,\cdot) \uu_n(t,\cdot),\uu_n(t,\cdot)\rangle dx \\
&+ 2\int_{B_n} \langle C(t,\cdot)\uu_n(t,\cdot),\uu_n(t,\cdot)\rangle dx.
\end{align*}
Thus, from Hypotheses $\ref{hyp_base}$(ii) and $\eqref{p-2-infty}$ we deduce that
\begin{align*}
D_t\|\uu_n(t,\cdot)\|_{2,n}^2 \le&  L_J \|\uu_n(t,\cdot)\|_{2,n}^2,
\end{align*}
whence $\|\uu_n(t,\cdot)\|_{2,n}^2=\|\G_n(t,s)\f\|_{2,n}^2\le e^{L_J(t-s)}\|\f\|^2_2$,
for any $(t,s)\in \Sigma_J$ and any $n \in \N$. This latter inequality together with estimate
\eqref{infty1} and the Riesz-Thorin interpolation theorem yields
\begin{equation*}
\|\G_n(t,s)\f\|_{p,k}\le e^{p^{-1}L_J(t-s)}\|\f\|_p
\end{equation*}
for any $(t,s)\in \Sigma_J$, $p \in [2,+\infty)$ and $k, n \in \N$ with $k\le n$.

Since $\G_n(t,s)\f$ converges pointwise to $\G(t,s)\f$ in $\Rd$ as $n \to +\infty$, Fatou's lemma yields that
$\|\G(t,s)\f\|_{p,k}\le e^{p^{-1}L_J(t-s)}\|\f\|_p$, for any $k \in \N$. Letting $k \to +\infty$ in the previous inequality and using
Fatou's lemma again we get the first part of the claim.

Now, let us suppose that Hypotheses $\ref{uni11}$ are satisfied, too.
Multiplying the differential equation $(D_r-\boldsymbol{\mathcal A}(r))\G_n(r,s){\bf f}$ $={\bf 0}$ by $\g \in C^2_c([s,t]\times B_n;\Rm)$
and integrating by parts with respect to $r$ and $x$ in $[s,t]\times B_n$, we easily deduce that,
for any $\f\in C^\infty_c(B_n;\R^m)$, the function $\vv_n(s,\cdot)=\G_n^*(t,s)\f$ is a weak solution of
the backward Dirichlet Cauchy problem
\begin{equation}\label{prob_approx_dual}
\left\{
\begin{array}{lll}
D_s\vv_n(s,x)=-(\boldsymbol{\mathcal A}^*(s)\vv_n)(s,x), &\qquad t>s, \,x\in B_n,\\
\vv_n(s,x)={\bf 0}, &\qquad t>s,\,x\in\partial B_n,\\
\vv_n(t,x)=\f(x), &\qquad x\in\overline{B_n},
\end{array}
\right.
\end{equation}
where
\begin{equation*}
\boldsymbol{\mathcal A}^*{\bf v}=\sum_{i,j=1}^dD_i(q_{ij}D_j{\bf v})
-\sum_{i=1}^d B_iD_i{\bf v}+\bigg (C- \sum_{k=1}^d D_k B_k\bigg ){\bf v}
\end{equation*}
for any smooth function ${\bf v}:\Rd\to \Rm$. Actually, by the duality theory developed in
\cite{friedman} (see, in particular, Theorem 9.5.5), $\vv_n$ is the unique
classical solution of problem
\eqref{prob_approx_dual} and from Hypotheses $\ref{uni11}$ it follows that
$\|\G_n^*(t,s)\f\|_\infty \le e^{\kappa_0(t-s)}\|\f\|_\infty$, for any $t>s$ and
$\f$ as above (see \cite{AddAngLor15Cou} and the Appendix).
We can then apply the arguments above to $\G^*_n(t,s)$,
showing that \eqref{est_p} holds true with $\G(t,s)\f$ replaced by $\G^*(t,s)\f$ for any $p \ge 2$.
Indeed, multiplying the differential equation in
\eqref{prob_approx_dual} by $\vv_n$ and integrating by parts in $B_n$, we get
\begin{align*}
D_s\|\vv_n(s,\cdot)\|_{2,n}^2=&  -2\int_{B_n}\langle \vv_n(s,\cdot),(\boldsymbol{\mathcal A}^*(s)\vv_n)(s,\cdot)\rangle dx\\
=&\int_{B_n}
\langle Q(s,\cdot) D_x\vv_n(s,\cdot),D_x\vv_n(s,\cdot) \rangle dx
+\sum_{i=1}^d\int_{B_n}\langle (D_i B_i)(s,\cdot) \vv_n(s,\cdot),\vv_n(s,\cdot)\rangle dx\\
&- 2\int_{B_n} \langle C(s,\cdot)\vv_n(s,\cdot),\vv_n(s,\cdot)\rangle dx\\
\ge&\int_{B_n}\lambda_{\sum_{i=1}^dD_iB_i-2C}(s,\cdot)|\vv_n(s,\cdot)|^2dx .
\end{align*}
Since $-\lambda_A=\Lambda_{-A}$ for any symmetric matrix $A$, from \eqref{p-2-infty} it follows that
\begin{equation}\label{dual}
D_r\|\vv_n(r,\cdot)\|_{2,n}^2\ge -L_J\|\vv_n(r,\cdot)\|_{2,n}^2
\end{equation}
for any $r \in (s,t)$ and $n \in \N$.
Integrating \eqref{dual} with respect to $r$ from $s$ to $t$ and taking the final condition in \eqref{prob_approx_dual} into account, we get
\[
\|\G_n^*(t,s)\f\|_{2,n}^2\le e^{L_J(t-s)}\|\f\|_2^2.
\]
Again, by the Riesz-Thorin theorem and the uniform estimate $\|\G_n^*(t,s)\f\|_\infty \le e^{\kappa_0(t-s)}\|\f\|_\infty$, we obtain
\[
\|\G^*_n(t,s)\f\|_{p,n}\le e^{\frac{1}{p}(L_J+\kappa_0(p-2))(t-s)}\|\f\|_p,
\]
for any $(t,s)\in \Sigma_J$ and $p \in [2,+\infty)$. Arguing as above and letting $n \to +\infty$ in the previous inequality we get
\begin{equation}\label{est_p_n_dual}
\|\G^*(t,s)\f\|_{p}\le e^{\frac{1}{p}(L_J+\kappa_0(p-2))(t-s)}\|\f\|_p
\end{equation}
for the same values of $t,s$ and $p$.

Now, fix $p \in [1,2)$ and $\f \in C_c(\Rd;\Rm)$. Then, from \eqref{est_p_n_dual}
\begin{align*}
\|\G(t,s)\f\|_{p}=& \sup\left\{\int_{\Rd}\langle\G(t,s)\f,\g\rangle dx: \g \in C^\infty_c(\Rd;\Rm),\,\,\|\g\|_{p'}\le 1\right\}\\[1mm]
\le & \|\f\|_p \sup\{\|\G^*(t,s)\g\|_{p'}: \g \in C^\infty_c(\Rd;\Rm),\,\,\|\g\|_{p'}\le 1\}\\[1mm]
\le & e^{\frac{1}{p'}(L_J+ \kappa_0(p'-2))(t-s)}\|\f\|_p
\end{align*}
for any $(t,s)\in\Sigma_J$, which completes the proof.
\end{proof}

The case when the pointwise estimate \eqref{pointwise} holds is much simpler. Indeed, estimate \eqref{est_p}
can be obtained just requiring conditions on the scalar evolution operator $G(t,s)$. As an immediate consequence
of estimate \eqref{pointwise} we get the following
\begin{thm}\label{th2}
Assume that Hypotheses $\ref{uni2}$ hold true and fix $p \in [1+\frac{1}{4\beta},+\infty)$. If $G(t,s)$ preserves $L^1(\Rd)$
and satisfies \eqref{est_p} with $p=m=1$ and $c_1=\tilde{c}_1$,  then estimate \eqref{est_p} holds true for any $(t,s)\in \Sigma_J$ and
$\f\in C_c(\Rd;\Rm)$ with $c_p(r)=e^{K_J r}\tilde{c}_1(r)$.
\end{thm}

\begin{rmk}{\rm
Sufficient conditions for the scalar evolution
operator $G(t,s)$ to satisfy \eqref{est_p} with $p\in [1,+\infty)$ can be found in
\cite[Thms. 5.3 \& 5.4]{AngLor10Com} when $\A$
is not in divergence form. Adapting the cited theorems to our case, one can show that estimate
\eqref{est_p} is satisfied with $p=1$
if there exists an interval $J\subset I$ and a positive constant $\Gamma_J$ such that either
 ${\rm{div}}_x b\ge- \Gamma_J$ or $ |b|^2\le \Gamma_{J}\lambda_Q$ in $J \times \Rd$.
 }\end{rmk}

\begin{prop}  \label{def_Lp}
Let the assumptions of Theorem \ref{th1} (resp. Theorem \ref{th2}) be satisfied. Then, the evolution operator $\G(t,s)$
associated with $\boldsymbol{\mathcal A}(t)$ in $C_c(\Rd;\Rm)$ admits
a continuous extension to $L^p(\Rd;\Rm)$ for any $p \in [1,+\infty)$ $($resp. $p \in [1+\frac{1}{4\beta},+\infty))$. Moreover, $\G(t,s)\f$
tends to $\f$ in $L^p(\Rd,\Rm)$ as $t \to s^+$, for any $s \in I$, $\f\in L^p(\Rd;\Rm)$ and $p \in [1,+\infty)$
$($resp. $p \in [1+\frac{1}{4\beta},+\infty))$.
\end{prop}
\begin{proof}
The first part of the claim is an easy consequence of estimate \eqref{est_p}.
Indeed, fix $(t,s)\in \Sigma_J$, $\f\in L^p(\Rd;\Rm)$ and let $(\f_n)$ be a sequence in $C_c(\Rd;\Rm)$
converging to $\f$ in $L^p(\Rd; \Rm)$, as $n \to +\infty$. Then, from \eqref{est_p} it follows that
\begin{equation}\label{Cau_seq}
\|\G(t,s)(\f_n-\f_k)\|_{L^p(\Rd;\Rm)}\le c_p(t-s)\|\f_n-\f_k\|_{L^p(\Rd;\Rm)}
\end{equation}
for any $n, k\in \N$ and, consequently, $(\G(t,s)\f_n)$ is a Cauchy sequence in $L^p(\Rd; \Rm)$.
We can then define $\G(t,s)\f$ as the $L^p(\Rd; \Rm)$-limit of $\G(t,s)\f_n$ as $n \to +\infty$.
Moreover, from \eqref{Cau_seq} it follows that $\|\G(t,s)\f\|_{L^p(\Rd;\Rm)}\le c\|\f\|_{L^p(\Rd;\Rm)}$ for any $\f\in L^p(\Rd;\Rm)$.

To prove the remaining part of the claim it suffices to show that, for any $t>s \in I$, any $x\in\Rd$ and any $\f \in C^2_c(\Rd;\Rm)$,
\begin{equation}\label{for}
(\G(t,s)\f)(x)- \f(x)= -\int_{s}^t (\G(t,r)\boldsymbol{\mathcal A}(r)\f)(x)dr.
\end{equation}
Indeed, fix $[a,b]\subset I$; from estimates \eqref{for} and \eqref{est_p} we deduce that
\begin{align*}
\|\G(t,s)\f-\f\|_{L^p(\Rd;\Rm)}\le \sup_{r\in[a,b]}\|\boldsymbol{\mathcal A}(r)\f\|_{L^p(\Rd;\Rm)}\int_s^t c_p(r-s)dr
\end{align*}
for any $s \in [a,b]$ and $t \ge s$. Since, in our assumptions, the last integral vanishes as $t \to s^+$, $\G(t,s)\f$ tends
to $\f$ in $L^p(\Rd;\Rm)$ as $t \to s^+$ and $s \in [a,b]$. A standard density argument and the arbitrariness of $[a,b]$
allow us to get the same result for $\f \in L^p(\Rd;\Rm)$ and any $s \in I$.

Let us show formula \eqref{for}. From \cite[Thm 2.3 (ix)]{Acq88Evo} (see also \cite[Thm. A.1]{AngLor14Non}),
we know that, for any $n$ such that ${\rm{supp}}(f)\subset B_n$,
\begin{equation}\label{for_n}
(\G_n^{\mathcal D}(t,s_1)\f)(x)- (\G_n^{\mathcal D}(t,s_0)\f)(x)= \int_{s_0}^{s_1} (\G_n^{\mathcal D}(t,r)\boldsymbol{\mathcal A}(r)\f)(x)dr
\end{equation}
for any $s_0 \le s_1 \le t$, $x \in \Rd$.
Since the function $\boldsymbol{\mathcal A}(r)\f$ belongs to $C_b(\Rd;\Rm)$, by Proposition \ref{recall}
$\G_n^{\mathcal D}(\cdot, r)\boldsymbol{\mathcal A}(r)\f$ converges to $\G(\cdot, r)\boldsymbol{\mathcal A}(r)\f$ in 
$C_{\rm loc}^{1,2}((r,+\infty)\times \Rd; \Rm)$. Thus, letting $n \to +\infty$ in \eqref{for_n} and choosing $s_1=t$ we get \eqref{for}.
\end{proof}

\section{Hypercontractivity estimates}

The aim of this section consists in proving that, under suitable assumptions,
the evolution operator ${\bf G}(t,s)$ maps $L^p(\Rd;\Rm)$ into $L^q(\Rd;\Rm)$ for any $t>s$
and $1 \le p \le q \le +\infty$ and that
\begin{equation}\label{iper}
\|{\bf G}(t,s)\f\|_{L^q(\Rd;\Rm)} \le c_{p,q}(t-s)\|\f\|_{L^p(\Rd;\Rm)}, \qquad\,\, t >s,\, \f \in L^p(\Rd; \Rm),
\end{equation}
for suitable functions $c_{p,q}:(0,+\infty)\to (0,+\infty)$.

\begin{thm}\label{thm_hyper}
Assume that Hypotheses $\ref{uni1}$ hold true and that, for some interval
$J \subset I$, estimate
\eqref{p-2-infty} is satisfied for any $(t,s)\in \Sigma_J$. Then, the following properties are satisfied.
\begin{enumerate}[\rm(i)]
\item Estimate \eqref{iper}
holds true for any $2 \le p \le q\le +\infty$, $(t,s)\in \Sigma_J$ and $\f
\in L^p(\Rd;\Rm)$. Moreover, $c_{2,\infty}(r)\le k_1 e^{k_2r}$
 for some positive $k_1$, $k_2$ depending on $m$, $d$, $\inf_{J\times
\Rd}\lambda_Q$, $L_J$, and
\footnote{Here and below $c_p,\ 1<p<\infty$, is the constant in Theorem \ref{th2}.}
$c_{p,q}(r)=(c_p(r))^{p/q}(c_{2,\infty}(r))^{2(q-p)/pq}$, for any $r>0$
and $(p,q)\neq (2,\infty)$.
\item If, in addition, Hypotheses $\ref{uni11}$ are satisfied, then estimate
\eqref{iper} holds true for any $1 \le p \le q\le +\infty$, $t,s$ and $\f$
as in $($i$)$.  Moreover, $c_{1,2}(r)\le k_1 e^{k_2r}$ for some positive $k_1$, $k_2$ as in (i) and
\[
c_{p,q}(r)=(c_p(r))^{\frac{p(2-q)}{q(2-p)}}(c_{1,2}(r))^{\frac{2(q-p)}{pq}}c_2^{4\frac{(q-p)(p-1)}{pq(2-p)}}
\]
for any $r>0$, if $q\le 2$, and $c_{p,q}(r)=c_{p,2}(r/2)c_{2,q}(r/2)$ for any $r>0$, if $p<2<q$.
\end{enumerate}
\end{thm}
\begin{proof}
Taking the result of the Proposition \ref{def_Lp} into account, we confine
ourselves to proving \eqref{iper} for functions belonging to $C_c(\Rd;\Rm)$.

(i) Fix $\f \in C_c(\Rd;\Rm)$ and let $J$ be as in the assumptions. Note
that it suffices to prove that
\begin{equation}\label{aim}
\|{\bf G}(t,s)\f\|_\infty\le c_{2,\infty}(t-s)\|\f\|_{L^2(\Rd;\Rm)},\qquad\;\, (t,s)\in
\Sigma_J
\end{equation}
for some positive function $c_{2,\infty}:(0,+\infty)\to (0,+\infty)$.
Indeed, once \eqref{aim} is proved, using the estimate $\|{\bf G}(t,s)\f\|_\infty\le \|\f\|_\infty$, which holds for any $t>s \in I$,
and the Riesz-Thorin theorem, we deduce that
$\|{\bf G}(t,s)\f\|_\infty\le c_{p,\infty}(t-s)\|\f\|_{L^p(\Rd;\Rm)}$ for any
$p \in [2,+\infty]$, $(t,s)\in \Sigma_J$ where
$c_{p,\infty}(t-s)=[c_{2,\infty}(t-s)]^{\frac{2}{p}}$ for any $p>2$.
On the other hand, Theorem \ref{th1} shows that $\|{\bf G}(t,s)\f\|_{L^p(\Rd;\Rm)}\le
c_p(t-s)\|\f\|_{L^p(\Rd;\Rm)}$,
for any $(t,s)\in \Sigma_J$ and $p \ge 2$. Hence, again by interpolation
we deduce that
\[
\|{\bf G}(t,s)\f\|_{L^q(\Rd;\Rm)}\le c_{p,q}(t-s)\|\f\|_{L^p(\Rd;\Rm)}, \qquad\;\, (t,s)\in \Sigma_J
\]
for any $2\le p \le q< +\infty$, where
$c_{p,q}(t-s)=[c_p(t-s)]^{\frac{p}{q}}[c_{p,\infty}(t-s)]^{1-\frac{p}{q}}$.

So, let us prove \eqref{aim}. First, observe that for any $n\in\N$, any $\h\in
C^2(\overline{B_n};\Rm)$, which vanishes on $\partial B_n$, and $\lambda>0$, it holds that
\begin{align*}
\int_{B_n} \langle \lambda \h-\boldsymbol{\mathcal A}^*(s)\h, \h\rangle dx=
&\sum_{i=1}^d\int_{B_n}\langle QD_xh_i,D_xh_i\rangle dx +\lambda
\|\h\|_2^2
+2^{-1} \sum_{i=1}^d\int_{B_n}{\rm Tr}(B_i D_i (\h\otimes \h)) dx
\\
&- \int_{B_n}\bigg\langle\bigg (C-\sum_{i=1}^d  D_iB_i\bigg ) \h,\h\bigg\rangle dx\\
&\ge \nu_0\|D_x \h\|_{L^2(B_n;\Rm)}^2+ \lambda\|\h\|_{L^2(B_n;\Rm)}^2 - \int_{B_n}\bigg\langle
\bigg (C-\frac{1}{2}\sum_{i=1}^d  D_iB_i\bigg ) \h,\h\bigg\rangle dx \\
&\ge    \nu_0\|D_x \h\|_{L^2(B_n;\Rm)}^2+ (\lambda-L_J/2)\|\h\|_{L^2(B_n;\Rm)}^2
\end{align*}
for any $s \in J$, with $L_J$ as in \eqref{p-2-infty}, where $\nu_0$ is the ellipticity bound in Hypotheses 
\ref{hyp_base}(ii). Nash's inequality (see
\cite[Thm. 2.4.6]{Dav89Hea}) together with the latter estimate yield
\begin{equation} \label{nash}
\int_{\Rd} \langle (\lambda -\boldsymbol{\mathcal A}^*(s))\h, \h\rangle dx
\ge c_1 \|\h\|_{W^{1,2}(B_n;\Rm)}^2\ge c_2\|\h\|_{L^2(B_n;\Rm)}^{2+4/d}\|\h\|_{L^1(B_n;\Rm)}^{-4/d}
\end{equation}
for any $\lambda >L_J/2$, $s \in J$ and some positive constants $c_1, c_2$
depending on $\nu_0, L_J$ and $m$.
Now, fix $\g \in C^\infty_c(\Rd;\Rm)$ and $\lambda >L_J/2$. For any $n \in
\N$, such that ${\rm supp}(f)\subset B_n$, we set
\[
v_n(s)=\|e^{-\lambda(t-s)}\G_n^*(t,s)\g\|_{L^2(B_n;\Rm)}^2, \qquad\;\, (t,s)\in \Sigma_J,
\]
where, as in the proof of Theorem \ref{th1}, $\G_n^*(t,s)\g$ denotes the unique classical solution of
\eqref{prob_approx_dual}.
Estimate \eqref{nash} implies
\begin{align}
v_n'(s)= & 2  e^{-2\lambda(t-s)}\int_{\Rd}\langle (\lambda
-\boldsymbol{\mathcal A}^*(s))\G_n^*(t,s)\g,\G_n^*(t,s)\g\rangle dx\notag \\
 \ge &   2 c_2  \|e^{-\lambda(t-s)}\G_n^*(t,s)\g\|_{L^2(B_n;\Rm)}^{2+4/d}
\|e^{-\lambda(t-s)}\G_n^*(t,s)\g\|_{L^1(B_n;\Rm)}^{-4/d}\notag\\
 \ge &  2 c_2 e^{\frac{4}{d}\lambda(t-s)}
\|e^{-\lambda(t-s)}\G_n^*(t,s)\g\|_{L^2(B_n;\Rm)}^{2+4/d}\|\g\|_{L^1(B_n;\Rm)}^{-4/d},
\label{dom-10}
\end{align}
where in the last inequality we have used the estimate
$\|{\bf G}^*_n(t,s)\g\|_{L^1(B_n;\Rm)} \le \|\g\|_{L^1(B_n;\Rm)}$ which holds true for any $\g \in C^\infty_c(\Rd;\Rm)$.
Indeed, the function ${\bf G}^*_n(t,s)\g$ belongs to $L^1(B_n;\R^m)$ and
\begin{align*}
\Big|\int_{B_n}\langle {\bf G}^*_n(t,s)\g,\f\rangle dx \Big| &=
\Big|\int_{B_n}\langle \g,{\bf G}_n(t,s)\f\rangle dx \Big|
\\
&\le \|\g\|_{L^1(B_n;\R^m)}\|{\bf G}_n(t,s)\f\|_{L^{\infty}(B_n;\R^m)}
\\
&\le \|\g\|_{L^1(B_n;\R^m)}\|\f\|_{L^{\infty}(B_n;\R^m)}
\end{align*}
for any $\f\in C_b(B_n;\R^m)$, since the proof of Proposition \ref{prop-appendix} shows that
$\|{\bf G}_n(t,s)\f\|_{L^{\infty}(B_n;\R^m)}\le\|\f\|_{L^{\infty}(B_n;\R^m)}$ for any $t\ge s$.
By approximating any $\f\in L^{\infty}(B_n;\R^m)$ by
a bounded sequence $(\f_n)\subset C_b(B_n;\R^m)$ converging to $\f$ in a dominated way, we conclude that
\[
\Big|\int_{B_n}\langle {\bf G}^*_n(t,s)\g,\f\rangle dx\Big|\leq \|\g\|_{L^1(B_n;\R^m)}\|\f\|_{L^{\infty}(B_n;\R^m)}
\] 
for any such
$\f$. This estimate shows that $\|{\bf G}^*_n(t,s)\g\|_{L^1(B_n;\R^m)}\le \|\g\|_{L^1(B_n;\R^m)}$, as claimed.

{}From \eqref{dom-10} it thus follows that
\[
\frac{d}{ds}[(v_n(s))^{-2/d}]\le - \frac{4c_2}{d}
e^{\frac{4}{d}\lambda(t-s)} \|\g\|_{L^1(B_n;\Rm)}^{-4/d},\qquad\,\, (t,s)\in \Sigma_J,
\]
whence, integrating from $s$ to $t$ and estimating
$\int_s^te^{\frac{4}{d}\lambda(t-r)} dr$ from below by $1$, we get
\[
(v_n(t))^{-2/d}-(v_n(s))^{-2/d}\le -\frac{4c_2}{d} \|\g\|_{L^1(B_n;\Rm)}^{-4/d}.
\]
Consequently,
$v_n(s)=\|e^{-\lambda(t-s)}{\bf G}^*_n(t,s)\g\|_{L^2(B_n;\Rm)}^2\le
d^{d/2}(4c_2)^{-d/2}\|\g\|_{L^1(B_n;\Rm)}^2,$
for any $(t,s)\in \Sigma_J$. Thus, we have established that
\[
\|{\bf G}^*_n(t,s)\g\|_{L^2(B_n;\Rm)} \le c_0 e^{\lambda(t-s)}\|\g\|_{L^1(B_n;\Rm)},
\]
for any $\g \in C_c(\Rd;\Rm)$, $(t,s)\in \Sigma_J$, $\lambda \ge L_J/2 $
and $c_0:=d^{d/4}(4c_2)^{-d/4}$. By duality, the latter inequality leads
to
\begin{align}  \label{app}
\|{\bf G}_n(t,s)\f\|_{\infty}  &= \sup\left\{\int_{\Rd}\langle \f , {\bf G}^*_n(t,s)\g\rangle dx: \g \in
C^\infty_c(B_n;\Rm),\ \|\g\|_{L^1(B_n;\Rm)} \le 1\right\}
\\ \nonumber
& \le c_0 e^{\lambda(t-s)}\|\f\|_{L^2(B_n;\Rm)}
\end{align}
for any $(t,s)\in \Sigma_J$. Letting $n \to +\infty$ in \eqref{app} yields
estimate \eqref{aim} with $c_{2,\infty}(t-s)=c_0 e^{\lambda(t-s)}$. \\
(ii) The second part of the statement can be easily obtained arguing again
by interpolation as in (i). In this case, since
$\|{\bf G}(t,s)\f\|_{L^p(\Rd;\Rm)}\le c_p(t-s)\|\f\|_{L^p(\Rd;\Rm)}$, for any $(t,s)\in \Sigma_J$
and $p \in [1,2]$, it is enough to prove that
\begin{equation} \label{aim_1-2}
\|{\bf G}(t,s)\f\|_{L^2(\Rd;\Rm)}\le c_{1,2}(t-s)\|\f\|_{L^1(\Rd;\Rm)}, \qquad\;\, (t,s)\in \Sigma_J,
\end{equation}
Once \eqref{aim_1-2} is proved, using Riesz-Thorin theorem and interpolating between \eqref{est_p}, with $p=2$,
and \eqref{aim_1-2}, we get \eqref{iper}
with $q=2$. Next, interpolating between this latter estimate and, again, \eqref{est_p}, we get \eqref{iper} for any $1\le p<q\le 2$, with
$c_{p,q}(r)=(c_p(r))^{\frac{2-q}{q(2-p)}}(c_{1,2}(r))^{\frac{2(q-p)}{pq}}$. Finally, splitting
${\bf G}(t,s)={\bf G}(t,(t+s)/2){\bf G}((t+s)/2,s)$, we get \eqref{iper} with $p<2<q$ and $c_{p,q}(r)=c_{p,2}(r/2)c_{2,q}(r/2)$.

The proof of \eqref{aim_1-2} can be obtained arguing as in (i) replacing
the function $v_n$ defined there by the function
$u_n(t)=\|e^{-\lambda(t-s)}{\bf G}(t,s)\g\|_{L^2(\Rd;\R^m)}^2$ for any $(t,s)\in \Sigma_J$.
\end{proof}

Theorem \ref{thm_hyper} can now be used to prove that the hypercontractivity estimate \eqref{iper} holds true also when
Hypotheses \ref{uni2} are satisfied, see also Remark \ref{alpha}.

\begin{thm}
Let us assume that Hypotheses $\ref{uni2}$ hold true and that for some interval $J\subset I$ there exist a
positive constant $\lambda_J$ and two functions $\kappa_J:J\times\Rd\to\R$, bounded from above, and
$\varphi_J\in C^2(\R^d)$, blowing up as $|x|\to +\infty$, such that ${\rm div}_x b+\kappa_J\ge 0$,
in $J \times \Rd$ and $\sup_{J\times\R^d}(\tilde{\A}\varphi_J-\lambda\varphi_J)<+\infty$,
where $\tilde{\A}={\rm div}(Q D_x)- \langle b,D_x\rangle+2\kappa_J$. Then, ${\bf G}(t,s)$ maps $L^p(\Rd;\Rm)$ into $L^q(\Rd;\Rm)$
for any $1+ \frac{1}{4\beta} \le p \le q \le +\infty$. Moreover,
$\|{\bf G}(t,s)\f\|_{L^q(\Rd;\Rm)}\le \tilde{c}_{p,q}(t-s)\|\f\|_{L^p(\Rd;\Rm)}$
for any $(t,s)\in \Sigma_J$, $1+\frac{1}{4\beta} \le p \le q \le +\infty$ and some function
${\tilde c}_{p,q}:(0,+\infty)\to (0,+\infty)$.
\end{thm}
\begin{proof}
Note that all the assumptions of Theorem \ref{thm_hyper}(ii) are satisfied by the scalar operator $\A$ in \eqref{defA}.
As a consequence, the evolution operator $G(t,s)$ associated with $\A$ satisfies \eqref{iper} for any $p,q$ as in the statement. In particular
$G(t,s)$ maps $L^1(\Rd)$ into $L^{q/p}(\Rd)$ and
\begin{equation}\label{iper_scalar}
\|G(t,s)\psi\|_{L^{q/p}(\Rd)}\le c_{1,q/p}(t-s)\|\psi\|_{L^1(\Rd)},
\qquad\,\; (t,s)\in \Sigma_J,\,\psi \in L^1(\Rd).
\end{equation}
Therefore, from \eqref{pointwise} and \eqref{iper_scalar} it follows that
\begin{align*}
\|{\bf G}(t,s)\f\|_{L^q(\Rd;\Rm)}^q &=\int_{\Rd}|{\bf G}(t,s)\f|^q dx \le e^{qK_p(t-s)/p}\int_{\Rd}(G(t,s)|\f|^p)^{q/p} dx
\\
& \le e^{qK_p(t-s)/p} [c_{1,q/p}(t-s)]^{q/p}\||\f|^p\|_{L^1(\Rd;\Rm)}^{q/p}
\\
&=e^{qK_p(t-s)/p}[c_{1,q/p}(t-s)]^{q/p} \|\f\|_{L^p(\Rd;\Rm)}^q
\end{align*}
for any $\f\in C_c(\Rd;\Rm)$ and $(t,s)\in \Sigma_J$. The density of $C_c(\Rd; \Rm)$ in $L^p(\Rd; \Rm)$ allows us to obtain the claim with
$\tilde{c}_{p,q}(r)= e^{K_pr/p}[c_{1,q/p}(r)]^{1/p}$, $r \ge 0$.
\end{proof}

\section{Pointwise gradient estimates}

In this section we prove some gradient estimates satisfied by the evolution operator ${\bf G}(t,s)\f$ when $\f \in C^\infty_c(\Rd;\Rm)$
when Hypotheses \ref{uni2} are satisfied. Notice that $p>1$ could be allowed in all the results if $\beta$ is arbitrary in \eqref{HpJ}, 
according to Remark \ref{alpha}. We also add the following assumptions.

\noindent
\begin{hyp}\label{gra_est}
There exist $\gamma\ge 1/4$ and a function $k$ such that
$|D_xq_{ij}|\le k\lambda_Q$ in $I\times \Rd$ for any $i,j=1, \ldots,d$ and
\begin{equation}\label{fin1}
\sup_{J \times \Rd} \bigg[\sqrt{d}m\xi\lambda_Q+\bigg(\sum_{i=1}^d |D_i C|^2\bigg)^{\frac{1}{2}}
+2\Lambda_C\bigg] < +\infty
\end{equation}
\begin{equation}\label{fin2}
\sup_{J \times \Rd} \bigg [\sqrt{d}\bigg (\sum_{i,j,l=1}^d|D_{il}q_{ij}|^2\bigg )^{\frac{1}{2}}
+\bigg (\sum_{i,j=1}^d |D_j\tilde{B}_i|^2\bigg )^{\frac{1}{2}}+\Lambda_{D_x b} + \Lambda_C +
M_{\gamma}\lambda_Q+\frac{1}{2}\bigg (\sum_{i=1}^d |D_i C|^2\bigg )^{\frac{1}{2}}\bigg ]<+\infty
\end{equation}
where $M_{\gamma}:=\gamma(\sqrt{d}m\xi+dk)^2+\frac{1}{2}\sqrt{d}m\xi+\frac{1}{4\gamma}$ (see Hypotheses $\ref{uni2}$).
\end{hyp}

\begin{thm}
\label{thm-avvvooooccato}
Assume that Hypotheses $\ref{uni2}$ (with $\sigma=1$) and Hypotheses $\ref{gra_est}$ are satisfied.
Then, for any $p\geq 1+\frac{1}{4(\beta\wedge\gamma)}$,
\begin{equation}\label{stima_grad1}
|D_x {\bf G}(t,s)\f|^p \le c_p\, e^{C_{p,J}(t-s)}G(t,s)(|\f|^p+|D\f|^p)
\end{equation}
for any $(t,s) \in \Sigma_J$, $\f \in C^\infty_c(\Rd;\Rm)$ and some positive constants $c_p$ and $C_{p,J}$, where $G(t,s)$ is
the evolution operator associated with $\A(t)$ in $C_b(\Rd)$.
\end{thm}

\begin{proof}
>From \cite[Prop. 2.4]{KunLorLun09Non} it follows that  $|G(t,s)\psi|^p\le G(t,s)|\psi|^p$,
for any $\psi\in C_b(\Rd)$, $t \ge s \in I$ and $p\in [1,+\infty)$.
Thus, it suffices to prove the claim only for $p \in [1+\frac{1}{4(\beta\wedge\gamma)},2]$.
Let $J$ be as in Hypotheses \ref{gra_est}, $\f \in C^\infty_c(\Rd;\Rm)$ and for large $n\in\N$, we consider
 the classical solution $\uu_n=\G_n^{\mathcal N}(\cdot,s)\f$ of the Cauchy-Neumann
problem \eqref{prob_approx_Neu}. The core of the proof consists in proving that
\begin{equation}\label{stima_grad1_appr}
|D_x \uu_n(t,\cdot)|^p \le e^{C_{p, J}(t-s)}G_n^{\mathcal N}(t,s)(|\f|^2+|D\f|^2)^{\frac{p}{2}}
\end{equation}
for any $(t,s)\in \Sigma_J$, $\f \in C^\infty_c(\Rd;\Rm)$, $p \in [1+\frac{1}{4(\beta\wedge\gamma)},2]$ and some
positive constant $C_{p,J}$. Here, $G_n^{\mathcal N}(t,s)$ denotes the evolution operator associated with the
restriction of $\A(t)$ (see \eqref{defA}) to $B_n$, with homogeneous Neumann boundary conditions.
Indeed, once \eqref{stima_grad1_appr} is proved, estimate \eqref{gra_est} follows, from Proposition \ref{recall},
with $c_p= 2^{(p/2-1)\vee 0}$.

So, let us prove \eqref{stima_grad1_appr}. For any $\varepsilon >0$, let us consider the function
$v_n= (|\uu_n|^2+|D_x \uu_n|^2+\varepsilon)^{\frac{p}{2}}$.
>From \cite[Thm. IV.5.5]{LadSolUra68Lin} it follows that
$v_n \in C^{1,2}([s,+\infty)\times \Rd)\cap C_b([s,T]\times \Rd)$ for any $T>s$. Moreover, $v_n$ solves the problem
\begin{equation}\label{pro_scal}
\left\{\begin{array}{ll}
D_t v_n- \A(t)v_n=pv_n^{1-2/p}\bigg (\displaystyle\sum_{i=1}^5\psi_i+(2-p)v_n^{-2/p}\psi_6\bigg ), \qquad\,\,  &(s,+\infty)\times B_n,\\
\displaystyle\frac{\partial v_n}{\partial \nu} \le 0\qquad\,\,  &(s,+\infty)\times \partial B_n,\\[1.6mm]
v_n(s)= (|\f|^2+|D_x \f|^2+\varepsilon)^{p/2} \qquad\,\,  & B_n,
\end{array}
\right.
\end{equation}
where
\begin{align*}
\psi_1= &\sum_{i,j,l=1}^d\sum_{k=1}^m D_{li}q_{ij}D_lu_{n,k}D_j u_{n,k}+\sum_{i,l=1}^d\sum_{k,j=1}^m D_l(\tilde{B}_i)_{kj}D_lu_{n,k}D_i u_{n,j}
\\
&+\sum_{j=1}^m\langle D_x bD_x u_{n,j},D_xu_{n,j}\rangle+ \sum_{i=1}^d\langle CD_i \uu_n,D_i \uu_n \rangle,
\\[1mm]
\psi_2= &\sum_{i,j,l=1}^d\sum_{k=1}^m D_lq_{ij}D_{ij} u_{n,k}D_lu_{n,k}+\sum_{i,l=1}^d\sum_{k,j=1}^m (\tilde{B}_i)_{kj}D_{li} u_{n,j}D_lu_{n,k},
\\[1mm]
\psi_3= &\sum_{i=1}^d\langle \uu_n, \tilde{B}_i D_i \uu_n\rangle +\sum_{l=1}^d\sum_{k,j=1}^m D_l C_{kj}u_{n,j}D_l u_{n,k},
\\[1mm]
\psi_4= &\langle C\uu_n,\uu_n \rangle,
\\[1mm]
\psi_5= &-\sum_{k=1}^m \langle Q D_x u_{n,k}, D_x u_{n,k}\rangle
- \sum_{i=1}^d\sum_{k=1}^m\langle Q D_x D_i u_{n,k}, D_x D_i u_{n,k}\rangle,
\\[1mm]
\psi_6=&-\sum_{i,j=1}^dq_{ij}\bigg (\langle\uu,D_i\uu\rangle+\sum_{l=1}^d\langle D_{il}\uu,D_l\uu\rangle\bigg )
\bigg (\langle\uu,D_j\uu\rangle+\sum_{m=1}^d\langle D_{jm}\uu,D_m\uu\rangle\bigg )
\end{align*}
and the boundary condition in \eqref{pro_scal} follows since the normal derivative of
$|D_x u_{n,k}|^2$ is nonpositive in $(s,+\infty)\times\partial B_n$
for any $k=1, \ldots,m$ (see e.g., \cite{BerFor04Gra,BerForLor07Gra}).

Using Hypotheses $\ref{uni2}$(i)-(ii) and the inequality $|D_xq_{ij}|\le k\lambda_Q$, we get the following
estimates for the functions $\psi_i$, for $i=1,2,3$:
\begin{align*}
\psi_1 \le& \bigg [\sqrt{d}\bigg(\sum_{i,j,l=1}^d |D_{li}q_{ij}|^2\bigg)^{1/2}
+\bigg(\sum_{i,l=1}^d |D_l\tilde{B}_i|^2\bigg)^{1/2}
+\Lambda_{D_x b}+ \Lambda_C\bigg ]|D_x \uu_n|^2\\[1mm]
\psi_2 \le& \bigg [\bigg (\sum_{i=1}^d |D_iQ|^2\bigg )^{1/2}
+\bigg [\bigg (\sum_{i=1}^d |\tilde{B}_i|^2\bigg )^{1/2}\bigg ] |D_x \uu_n||D_x^2 \uu_n|
\le a(dk +\sqrt{d}m\xi)^2 \lambda_Q |D_x^2 \uu_n|^2+ \frac{1}{4a}\lambda_Q |D_x \uu_n|^2,
\\[1mm]
\psi_3 \le &\frac{1}{2}\bigg [\sqrt{d} m \xi\lambda_Q + \bigg (\sum_{i=1}^d |D_i C|^2\bigg )^{1/2}\bigg ](|\uu_n|^2+|D_x \uu_n|^2)
\end{align*}
in $J \times \Rd$. To estimate $\psi_6$, we observe that
\begin{align*}
\psi_6=&\sum_{h,k=1}^m\sum_{i,j=1}^dq_{ij}\bigg (u_{n,h}D_iu_{n,h}+\sum_{l=1}^dD_{il}u_{n,h}D_lu_{n,h}\bigg )
\bigg (u_{n,k}D_ju_{n,k}+\sum_{m=1}^dD_{jm}u_{n,k}D_mu_{n,k}\bigg )\\
=&\sum_{h,k=1}^mu_{n,h}u_{n,k}\langle QD_xu_{n,h},D_xu_{n,k}\rangle+2\sum_{h,k=1}^mu_{n,h}
\sum_{l=1}^dD_lu_{n,k}\langle QD_xu_{n,h},D_xD_lu_{n,k}\rangle\\
&+\sum_{h,k=1}^m\sum_{l,m=1}^dD_lu_{n,h}D_mu_{n,k}\langle QD_xD_lu_{n,h},D_xD_mu_{n,k}\rangle.
\end{align*}
It thus follows that
\begin{align*}
\psi_6\le &\bigg (\sum_{h=1}^m|u_{n,h}||Q^{1/2}D_xu_{n,h}|\bigg )^2+
2\sum_{h,k=1}^m|u_{n,h}||Q^{1/2}D_xu_{n,h}|\sum_{l=1}^d|D_lu_{n,k}||Q^{1/2}D_xD_lu_{n,k}|
\\
&+\sum_{h,k=1}^m\sum_{l,m=1}^d|D_lu_{n,h}||D_mu_{n,k}||Q^{1/2}D_xD_lu_{n,h}||Q^{1/2}D_xD_mu_{n,k}|
\\
\le & |\uu_n|^2\sum_{k=1}^d\langle QD_xu_{n,k},D_xu_{n,k}\rangle
\\
&+2|\uu_n||D_x\uu_n|
\bigg (\sum_{k=1}^d\langle QD_xu_{n,k},D_xu_{n,k}\rangle\bigg )^{\frac{1}{2}}
\bigg (\sum_{i=1}^d\sum_{k=1}^m\langle QD_xD_iu_{n,k},D_xD_iu_{n,k}\rangle\bigg )^{\frac{1}{2}}
\\
&+|D_x\uu_n|^2\sum_{i=1}^d\sum_{k=1}^m\langle QD_xD_iu_{n,k},D_xD_iu_{n,k}\rangle
\\
=& \bigg [|\uu_n|\bigg (\sum_{k=1}^d\langle QD_xu_{n,k},D_xu_{n,k}\rangle\bigg )^{\frac{1}{2}}+
|D_x\uu_n|\bigg (\sum_{i=1}^d\sum_{k=1}^m\langle QD_xD_iu_{n,k},D_xD_iu_{n,k}\rangle\bigg )\bigg ]^2
\\
\le & (|\uu_n|^2+|D_x\uu_n|^2)\bigg (\sum_{k=1}^d\langle QD_xu_{n,k},D_xu_{n,k}\rangle+
\sum_{i=1}^d\sum_{k=1}^m\langle QD_xD_iu_{n,k},D_xD_iu_{n,k}\rangle\bigg )
\\
\le &v^{\frac{2}{p}}_n\bigg (\sum_{k=1}^d\langle QD_xu_{n,k},D_xu_{n,k}\rangle+
\sum_{i=1}^d\sum_{k=1}^m\langle QD_xD_iu_{n,k},D_xD_iu_{n,k}\rangle\bigg ).
\end{align*}
Putting everything together, we get
\begin{align*}
\sum_{i=1}^5\psi_i+(2-p)\psi_6v_n^{-2/p}\le &\bigg [\sqrt{d}\bigg (\sum_{i,j,l=1}^d|D_{il}q_{ij}|^2\bigg )^{1/2}
+\bigg(\sum_{i,j=1}^d |D_j\tilde{B}_i|^2\bigg )^{1/2}+\Lambda_{D_xb}+\Lambda_C
\\
&+\bigg (\frac{1}{4a}+p-1+\frac{1}{2}\sqrt{d}m\xi\bigg )\lambda_Q
+\frac{1}{2}\bigg(\sum_{i=1}^d |D_i C|^2\bigg)^{1/2}\bigg ]|D_x\uu_n|^2\\
&+[a(dk+\sqrt{d}m\xi)^2-(1-p)]\lambda_Q|D_x^2\uu_n|^2
\\
&+\bigg \{\Lambda_C+\frac{1}{2}\bigg [\sqrt{d} m \xi\lambda_Q + \bigg(\sum_{i=1}^d |D_i C|^2\bigg )^{1/2}\bigg ]\bigg\}|\uu_n|^2
\end{align*}
for any $a=a(t)$ and, choosing $a =(p-1)(dk +\sqrt{d}m\xi)^{-2}$, we conclude that
\begin{align*}
\sum_{i=1}^5\psi_i+(2-p)\psi_6v_n^{-2/p}\le & \bigg [\sqrt{d}\bigg (\sum_{i,j,l=1}^d|D_{il}q_{ij}|^2\bigg )^{1/2}
+\bigg(\sum_{i,j=1}^d |D_j\tilde{B}_i|^2\bigg )^{1/2}+\Lambda_{D_x b}+\Lambda_C +M_{\gamma}\lambda_Q
\\
&+ \frac{1}{2}\bigg (\sum_{i=1}^d |D_i C|^2\bigg )^{1/2} \bigg ]|D_x \uu_n|^2
\\
&+\bigg [\frac{1}{2}\sqrt{d}m \xi\lambda_Q+\frac{1}{2}\bigg(\sum_{i=1}^d |D_i C|^2\bigg )^{1/2}
+\Lambda_C\bigg]|\uu_n|^2
\end{align*}
in $J\times \Rd$.
Using estimates \eqref{fin1} and \eqref{fin2} we conclude that $D_t v_n- \A(t)v_n\le C_{p,J}v_n$ in $J\times \Rd$
for some positive constant $C_{p,J}$. Hence, the function
$w_n(t,\cdot)= v_n(t,\cdot)- e^{C_{p,J}(t-s)}G_n^{\mathcal N}(t,s)(|\f|^2+|D\f|^2+\varepsilon)^{p/2}$ solves the problem
\begin{equation*}
\left\{\begin{array}{ll}
D_t w_n-( \A(t)+C_{p,J})w_n \le 0, \qquad\,\,  &(s,T]\times B_n,\\[1mm]
\displaystyle\frac{\partial w_n}{\partial \nu} \le 0,\qquad\,\,  &(s,T]\times \partial B_n,\\[1.5mm]
w_n(s)=0, \qquad\,\,  & B_n.
\end{array}
\right.
\end{equation*}
The classical maximum principle yields that $w_n\le 0$ in $(s,T)\times B_n$,
whence, letting $\varepsilon \to 0^+$, estimate \eqref{stima_grad1_appr} follows at once.
\end{proof}

\begin{thm}\label{Lp-w1p-thm}
Assume that Hypotheses \ref{uni2} (with $\sigma=1$) and Hypotheses \ref{gra_est} are satisfied with $J=I$.
If $\Lambda_C \le -2\gamma dm^2 \xi^2\lambda_Q$ in $I \times \Rd$, where $\gamma$ is as in Hypotheses $\ref{gra_est}$, then the estimate
\begin{equation}\label{lp-w1p}
|D_x {\bf G}(t,s)\f|^p\le k_pe^{h_p(t-s)}(t-s)^{-\frac{p}{2}}G(t,s)|\f|^p,
\end{equation}
holds in $\Sigma_I\times\Rd$, for any $p \in [1+\frac{1}{4(\beta \wedge \gamma)},+\infty)$, $\f\in C_c^\infty(\Rd,\Rm)$
and some positive constants $k_p$ and $h_p$.
\end{thm}
\begin{proof}
Using the same arguments as in the proof of Theorem \ref{thm-avvvooooccato} we can limit ourselves to proving \eqref{lp-w1p}
when $p\in [1+\frac{1}{4(\beta\wedge\gamma)},2]$. Note that, under our assumptions, the estimates \eqref{pointwise} and \eqref{stima_grad1}
hold true for any $p \in [1+\frac{1}{4(\beta\wedge\gamma)},2]$, $\f\in C_c^\infty(\Rd,\Rm)$
and $t>s\in I$, with positive constants $K_J$ in \eqref{pointwise} and $C_p$ in \eqref{stima_grad1}, independent of $J$. Moreover, after a rescaling
argument we can assume that $K_J<0$. Thus, for any fixed $p \in [1+\frac{1}{4(\beta\wedge\gamma)},2]$, $\f\in C_c^\infty(\Rd,\Rm)$,
from \eqref{stima_grad1} and the evolution law it follows that
\begin{align*}
|D_x {\bf G}(t,s)\f|^p&=|D_x {\bf G}(t,\sigma){\bf G}(\sigma,s)\f|^p
\\
&\le c_pe^{C_p(t-\sigma)}G(t,\sigma)[|{\bf G}(\sigma,s)\f|^p+|D_x{\bf G}(\sigma,s)\f|^p]
\\
& \le c_pe^{C_p(t-\sigma)} \left[G(t,s)|\f|^p+G(t,\sigma)|D_x{\bf G}(\sigma,s)\f|^p\right]
\end{align*}
for any $\sigma \in (s,t)$. Since the transition kernel $p_{t,s}(x,y)$ associated with the evolution operator $G(t,s)$
is a positive $L^1$-function with respect to the variable $y$ with
$L^1$-norm equal to one (see \cite[Prop. 2.4]{KunLorLun09Non}), using the H\"older inequality we can estimate
\begin{align*}
G(t,\sigma)|D_x{\bf G}(\sigma,s)\f|^p
=&G(t,\sigma)\left[|D_x{\bf G}(\sigma,s)\f|^p(|{\bf G}(\sigma,s)\f|^2+\delta)^{\frac{p(p-2)}{4}}
(|{\bf G}(\sigma,s)\f|^2+\delta)^{\frac{p(2-p)}{4}}\right]
\\
\le&\Big (G(t,\sigma)(|D_x{\bf G}(\sigma,s)\f|^2(|{\bf G}(\sigma,s)\f|^2+\delta)^{\frac{p-2}{2}})\Big)^{\frac{p}{2}}
\left(G(t,\sigma)(|{\bf G}(\sigma,s)\f|^2+\delta)^{\frac{p}{2}}\right)^{\frac{2-p}{2}}
\\
\le& \varepsilon^{\frac{2}{p}}\frac{p}{2}G(t,\sigma)\Big (|D_x{\bf G}(\sigma,s)\f|^2(|{\bf G}(\sigma,s)\f|^2+\delta)^{\frac{p-2}{2}}\Big )
\\
&+\bigg (1-\frac{p}{2}\bigg )\varepsilon^{\frac{2}{p-2}}G(t,\sigma)(|{\bf G}(\sigma,s)\f|^2+\delta)^{\frac{p}{2}}
\end{align*}
for any $\varepsilon, \delta >0$, whence
\begin{align*}
e^{-C_p(t-\sigma)}|D_x {\bf G}(t,s)\f|^p \le & 
c_pG(t,s)|\f|^p+c_p\left(1-\frac{p}{2}\right)\varepsilon^{\frac{2}{p-2}}G(t,\sigma)(|{\bf G}(\sigma,s)\f|^2+\delta)^{\frac{p}{2}}\nonumber
\\
&+\frac{p}{2}c_p\varepsilon^{\frac{2}{p}}G(t,\sigma)\left(|D_x{\bf G}(\sigma,s)\f|^2(|{\bf G}(\sigma,s)\f|^2+\delta)^{\frac{p-2}{2}}\right).
\end{align*}
Integrating the previous estimate with respect to $\sigma \in (s,t)$, we deduce
\begin{align}\label{pre}
|D_x {\bf G}(t,s)\f|^p \le \frac{C_pc_p}{1-e^{-C_p(t-s)}}
\bigg\{&(t-s) G(t,s)|\f|^p+\bigg(1-\frac{p}{2}\bigg )\varepsilon^{\frac{2}{p-2}}
\int_s^tG(t,\sigma)(|{\bf G}(\sigma,s)\f|^2+\delta)^{\frac{p}{2}}d\sigma
\nonumber\\
&+\frac{p}{2}\varepsilon^{\frac{2}{p}}\int_s^tG(t,\sigma)
\left(|D_x{\bf G}(\sigma,s)\f|^2(|{\bf G}(\sigma,s)\f|^2+\delta)^{\frac{p-2}{2}}\right) d\sigma\bigg\}.
\end{align}
The claim reduces to proving that there exists a positive constant $k_p$ such that
\begin{equation}\label{aim_int}
 \int_s^t G(t,\sigma)\left(|D_x{\bf G}(\sigma,s)\f|^2(|{\bf G}(\sigma,s)\f|^2+\delta)^{\frac{p-2}{2}}\right)d \sigma 
 \le k_p G(t,s)(|\f|^2+\delta)^{\frac{p}{2}}
\end{equation}
for any $(t,s)\in \Sigma_I$. Indeed, once \eqref{aim_int} is proved, we replace \eqref{aim_int} into \eqref{pre} and, 
using \cite[Prop. 3.1]{KunLorLun09Non},
we let $\delta \to 0^+$. Finally, using again \eqref{pointwise} to estimate
$G(t,\sigma)|{\bf G}(\sigma,s)\f|^p\le G(t,\sigma)G(\sigma,s)|\f|^p=G(t,s)|f|^p$,
we get
\begin{align*}
|D_x {\bf G}(t,s)\f|^p \le &\frac{C_pc_p}{1-e^{-C_p(t-s)}}\left\{\left[1+\bigg(1-\frac{p}{2}\bigg)
\varepsilon^{\frac{2}{p-2}}\right](t-s)+ \frac{p}{2}\varepsilon^{\frac{2}{p}}k_p\right\} G(t,s)|\f|^p
\end{align*}
and, minimising on $\varepsilon$,
\begin{align*}
|D_x {\bf G}(t,s)\f|^p \le  \frac{C_pc_p}{1-e^{-C_p(t-s)}}\left[(t-s)+k_p^{\frac{p}{2}}(t-s)^{1-\frac{p}{2}}\right] G(t,s)|\f|^p
\end{align*}
whence the claim follows.
Therefore, to conclude we prove \eqref{aim_int}.
To this aim, we set
\[
\psi_n(\sigma)= G_n^{\mathcal N}(t,\sigma)\left(|{\bf G}_n^{\mathcal N}(\sigma,s)\f|^2+\delta\right)^{\frac{p}{2}}
=G_n^{\mathcal N}(t,\sigma)\left(|{\bf u}_n(\sigma, \cdot)|^2+\delta\right)^{\frac{p}{2}}= G_n^{\mathcal N}(t,\sigma)(v_n(\sigma, \cdot))
\]
for any $\sigma \in [s,t]$ and $n \in \N$, where $G_n^{\mathcal N}(t,\sigma)$ and ${\bf G}_n^{\mathcal N}(t,\sigma)$
are the same evolution operator considered in the proof of Theorem \ref{thm-avvvooooccato}.
Since the normal derivative of the function $v_n(\sigma,\cdot)$ vanishes of $\partial B_n$ for any $\sigma\in (s,t)$,
classical results on evolution operators show that the function $\psi_n$ is differentiable in $(s,t)$ and a straightforward computation yields 
\begin{align*}
\psi'_n(\sigma)&= G_n^{\mathcal N}(t,\sigma)\left[D_\sigma v_n(\sigma, \cdot)-\mathcal{A}(\sigma) v_n(\sigma, \cdot)\right]
\\
&=pG_n^{\mathcal N}(t,\sigma)\left[(v_n(\sigma))^{1-\frac{2}{p}} \left(\sum_{i=1}^d\langle \uu_n, \tilde{B}_i D_i \uu_n\rangle+ 
\langle \uu_n, C\uu_n\rangle-\sum_{i,j=1}^dq_{ij}\langle D_i\uu,D_j\uu\rangle\right)\right. 
\\
&\left. \qquad\qquad\quad\;\,+ (2-p)(v_n(\sigma))^{1-\frac{4}{p}}\sum_{i,j=1}^d q_{ij}\langle \uu, D_i \uu\rangle \langle \uu, D_j \uu\rangle\right].
\end{align*}
Using \eqref{star}, we get
\begin{align*}
\psi'_n(\sigma) & \le pG_n^{\mathcal N}(t,\sigma)
\left[(v_n(\sigma))^{1-\frac{2}{p}} 
\left(\sum_{i=1}^d\langle \uu_n, \tilde{B}_i D_i \uu_n\rangle+ \langle \uu_n, C\uu_n\rangle+(1-p)\lambda_Q|D_x \uu_{n}|^2\right)\right].
\end{align*}
Thus, taking Hypotheses \ref{uni2}(i) into account, we deduce
\begin{align*}
\sum_{i=1}^d\langle \uu_n, \tilde{B}_i D_i \uu_n\rangle+ \langle \uu_n, C\uu_n\rangle\le& m\xi\lambda_Q|\uu_n|
\sum_{i=1}^d|D_i\uu_n|+ \Lambda_C|\uu_n|^2
\\
\le &(\varepsilon dm^2\xi^2)\lambda_Q|D_x\uu_n|^2+\bigg (\frac{\lambda_Q}{4\varepsilon}+\Lambda_C\bigg )|\uu_n|^2
\end{align*}
for any $\varepsilon=\varepsilon(t)>0$. Consequently,
\begin{align*}
\psi'_n(\sigma) & \le pG_n^{\mathcal N}(t,\sigma)\left[(v_n(\sigma))^{1-\frac{2}{p}} 
\left((\varepsilon dm^2\xi^2+1-p)\lambda_Q|D_x\uu_n|^2+\bigg (\frac{\lambda_Q}{4\varepsilon}+\Lambda_C\bigg )|\uu_n|^2\right)\right].
\end{align*}
Choosing $\varepsilon= (p-1)(2dm^2\xi^2)^{-1}$ implies
\begin{equation}\label{final}
\psi'_n(\sigma)
\le 2^{-1}p(1-p)\nu_0 G_n^{\mathcal N}(t,\sigma)\left[(v_n(\sigma))^{1-\frac{2}{p}}|D_x\uu_n|^2\right]
\end{equation}
Integrating both sides of \eqref{final} with respect to $\sigma$ in $[s+h,t-h]$ and then letting $n$ to $+\infty$ and 
$h$ to $0$ we get \eqref{aim_int} with $k_p=2[p(p-1)\nu_0]^{-1}$. The proof is so completed.
\end{proof}

\begin{coro}
Under the same Hypotheses as in Theorem $\ref{Lp-w1p-thm}$ and assuming
that $G(t,s)$ satisfies estimate \eqref{est_p} with $p=1$, the evolution
operator $\G(t,s)$ is bounded from $W^{\theta_1,p}(\Rd;\R^m)$ in
$W^{\theta_2,p}(\Rd;\R^m)$,
for any $p\in [1+\frac{1}{4(\beta\wedge \gamma)},+\infty)$, $0 \le
\theta_1\le \theta_2\le 1$ and $(t,s)\in \Sigma_I$.
\end{coro}

\begin{proof}
>From Theorem \ref{th2} it follows that $\|{\bf G}(t,s)\f\|_p \le
c_p(t-s)\|\f\|_p$ for any $t>s \in I$, $\f \in C^\infty_c(\Rd; \Rm)$ and
some positive function $c_p:(0,+\infty)\to (0,+\infty)$. Moreover,
integrating the estimates \eqref{stima_grad1} and \eqref{lp-w1p} in $\Rd$,
writing \eqref{est_p} with $p=1$ and $G(t,s)$ instead of $\G(t,s)$ and using the above 
estimate for $\|{\bf G}(t,s)\f\|_p$, it follows that
\begin{equation}\label{fragole}
\|{\bf G}(t,s)\f\|_{W^{1,p}(\Rd;\R^m)}\le c_p^1(t-s)\|\f\|_{W^{1,p}(\Rd;\R^m)},\qquad\;\,\|{\bf
G}(t,s)\f\|_{W^{1,p}(\Rd;\R^m)}\le c_p^2(t-s)\|\f\|_{L^p(\Rd;\R^m)},
\end{equation}
for any $t>s \in I$, $p\in [1+\frac{1}{4(\beta\wedge \gamma)},+\infty)$,
$\f \in C^\infty_c(\Rd; \Rm)$ and some positive functions
$c_p^i:(0,+\infty)\to (0,+\infty)$, $i=1,2$. By density, the first estimate in
\eqref{fragole} can be extended to any $\f \in W^{1,p}(\Rd;\Rm)$ and the second to $\f\in L^p(\Rd;\Rm)$.
Thus, the claim is proved for $\theta_2=1$ and $\theta_1=0,1$. The
remaining cases follows by interpolation, taking into account that for any
$\theta \in (0,1)$ and $p \in [1, +\infty)$, $W^{\theta,p}(\Rd;\Rm)$
equals the real interpolation space $(L^p(\Rd;\Rm);
W^{1,p}(\Rd;\Rm))_{\theta,p}$ with equivalence of the respective norms (see \cite[Thm. 2.4.1(a)]{triebel}).
\end{proof}

\section{Examples}

Here we exhibit some classes of elliptic operators to which Theorem
\ref{th1} can be applied.
Indeed examples of operators which satisfy the hypotheses of Theorem
\ref{th2} can be found in \cite{AngLor10Com}.
\begin{example}{\rm
Let $\boldsymbol{\mathcal A}$ be as in \eqref{operat-A} with $Q=I_m$,
$B_i(x)=-x_i(1+|x|^2)^a\hat{B}_i$ and $C(x)=-|x|^2(1+|x|^2)^b\hat{C}$
for any $x \in \Rd$, $i=1,\ldots,d$. Here, $\hat B_i$ ($i=1,\ldots,d$) and
$\hat C$ are constant,
symmetric and positive definite matrices and $b>2a\ge 0$. It is easy to
check that
\begin{equation*}
\mathcal{K}_\eta(x)\ge  -(1+|x|^2)^{2a}\sum_{i=1}^d x_i^2|\hat B_i|^2
+4|x|^2(1+|x|^2)^b \lambda_{\hat{C}}
\end{equation*}
for any $x\in\Rd$.
Moreover, choosing $\kappa(x)=-|x|^c$ with $c\in (2+2a,2+2b)$, we get 
\begin{align*}
\tilde{\mathcal{K}}_\eta(x)\ge& -(1+|x|^2)^{2a}\sum_{i=1}^d x_i^2|\hat
B_i|^2+4|x|^2(1+|x|^2)^b\lambda_{\hat{C}}-4(1+|x|^2)^a\sum_{i=1}^d
\Lambda_{\hat{B}_i}\\
&-8a(1+|x|^2)^{a-1}\sum_{i=1}^d \Lambda_{\hat{B}_i}x_i^2-4|x|^c
\end{align*}
for any $x\in\Rd$.
Since $b>2a$ and $c <2+ 2b$, the functions $\mathcal{K}_\eta$ and
$\tilde{\mathcal{K}}_\eta$ blow
 up at infinity as $|x|\to \infty$, uniformly with respect to $\eta \in
\partial B_1$. Therefore, assumption \eqref{weak} is satisfied both by 
$\mathcal{K}_\eta$ and $\tilde{\mathcal{K}}_\eta$.
On the other hand, taking into account that $c >2+2a$, the function
$\varphi(x)=1+|x|^2$, $x \in \Rd$, satisfies
Hypotheses \ref{uni1}(ii) and \ref{uni11}(ii) for any $\lambda>0$.
Finally, a straightforward computation shows that
\begin{align*}
\Lambda_{2C-\sum_{i=1}^d D_i B_i}(x)\le -2 |x|^2(1+|x|^2)^b \lambda_{\hat
C}+(1+|x|^2)^a \sum_{i=1}^d
\Lambda_{\hat{B}_i}+2a(1+|x|^2)^{a-1}\sum_{i=1}^d x_i^2
\Lambda_{\hat{B}_i}
\end{align*}
for any $x\in\Rd$.
The choice of $a$ and $b$ yields that estimate \eqref{p-2-infty} is
satisfied, too. Since,
all the assumptions in
Theorem \ref{th1} are satisfied, the evolution operator
${\bf G}(t,s)$
associated with $\boldsymbol{\mathcal A}$ is well-defined in $L^p(\Rd;
\Rm)$ for any $p \ge 1$. Moreover, estimate \eqref{est_p}
holds true, where $c_p(t-s)$ is defined in Theorem \ref{th1}.
}\end{example}

In the following example we consider the operator $\boldsymbol{\mathcal A}$ with $B_i$, $C$ as above, but allow the 
diffusion coefficients $q_{ij}$ to be unbounded as well. 
\begin{example}{\rm
Let $\boldsymbol{\mathcal A}$ be as in \eqref{operat-A} with
$Q(x)=(1+|x|^2)^\delta I_m$,
$B_i(x)=-x_i(1+|x|^2)^a I_m+(1+|x|^2)^b\hat{B}_i$ $(i=1,\ldots,d)$ and
$C(x)=-(1+|x|^2)^c\hat{C}$ for any $x \in \Rd$. We assume that
$\hat{B}_i$ ($i=1,\ldots,d$) and $\hat{C}$ are constant, symmetric and
positive definite matrices.
Finally, $\delta,a,b \in [0,+\infty)$ satisfy $2b \le \delta <a+1$ and $c
> 2a\vee (a+1)$.
We have that
\begin{align*}
\mathcal{K}_\eta(x) = (1+|x|^2)^{-\delta+2b}\sum_{i=1}^d\left[\langle
\hat{B}_i \eta, \eta\rangle^2
-|\hat{B}_i\eta|^2\right] +4(1+|x|^2)^c \langle \hat{C} \eta, \eta \rangle,
\end{align*}
for any $x\in \Rd$ and $\eta \in \partial B_1$. Since $\delta \ge 2b$, the
first term in the previous formula is bounded in $\Rd$,
therefore \eqref{weak} is clearly satisfied by $\mathcal K_\eta$ and also
by $\tilde{\mathcal{K}}_\eta$, where $\kappa(x)=-|x|^s$ and
$s\in(2+2a,2c)$.
Indeed,
\begin{align*}
\tilde{\mathcal{K}}_\eta(x) \ge\mathcal{K}_\eta(x)- 4(1+|x|^2)^a -
8b|x|^2(1+|x|^2)^{a-1}+8b(1+|x|^2)^{b-1}\sum_{i=1}^d x_i \langle \hat{B}_i \eta,\eta\rangle-|x|^s
\end{align*}
for any $x\in\Rd$.
The choice of $\delta ,a$, $b$ and $s$  yields that the function
$\varphi$, defined in (i) is a Lyapunov function
in $\R^d$ for both $\mathcal{A}$ and $\tilde{\mathcal A}$.
Moreover,
\begin{align*}
\Lambda_{2C-\sum_{i=1}^d D_i B_i}(x)\le - 2(1+|x|^2)^c \lambda_{\hat C}+
(1+|x|^2)^b+ 2b(1+|x|^2)^{b-1}|x|^2+ 2c(1+|x|^2)^{c-1}\sum_{i=1}^d |x_i|\Lambda_{\hat{B}_i},
\end{align*}
and, since the leading term in the previous estimate is the first term in
the right-hand side,
estimate \eqref{p-2-infty} is clearly satisfied. Thus, Theorem \ref{th1}
can be applied. Moreover,
since $c > \delta$, $2c> 2b-1$ and $b \le \delta$, the assumptions of Theorems
\ref{thm-avvvooooccato} and \ref{Lp-w1p-thm} are satisfied and
estimates \eqref{stima_grad1} and \eqref{lp-w1p} hold true in $\Rd$ for any $(t,s)\in \Sigma_I$.
}\end{example}

\begin{rmk}{\rm In the previous examples we can replace the constant matrices $I_m$,
$\hat{B}_j$ ($j=1, \ldots,d$) and $\hat C$
by matrices of the same type, i.e., by ${\rm diag}(q_i(t))$,
$\hat{B}_j(t)$ ($i=1, \ldots,m$, $j=1, \ldots,d$) and $\hat C(t)$
respectively,
whose entries are functions which belong to $C^{\alpha/2}_{ \rm
loc}(I)\cap C_b(I)$ and such that
$q_i$, $\lambda_{\hat{B}_i}$ ($i=1, \ldots,m$, $j=1, \ldots,d$) and
$\lambda_{\hat{C}}$, have positive infima on $I$.
}\end{rmk}

\appendix\section{Uniform estimates}

Now, we prove that the $L^\infty$-norm of the classical solutions of the Cauchy problems
\eqref{eq:cauchy_problem_system} and \eqref{prob_approx_dual}
can be estimated in terms of the $L^\infty$-norm of the initial datum. The proof of this result
can be found in \cite{AddAngLor15Cou} in the case when $\boldsymbol{\mathcal A}$ is not in divergence form.

\begin{prop}
\label{prop-appendix}
Let us assume that Hypotheses $\ref{hyp_base}$ hold true. If there exists a function $h:I\times \Rd\to \R$
bounded from above, such that
Hypotheses $\ref{uni1}$ are satisfied with $\mathcal{K}_\eta$ replaced by $\mathcal{K}_\eta+4 h$
and $\mathcal A_\eta$ replaced by $\mathcal A_\eta+  2h$
then the evolution operator associated with $\boldsymbol{\mathcal A}$ in $C_b(\Rd; \Rm)$ satisfies the estimate
\begin{equation*}
\|{\bf G}(t,s)\f\|_\infty \le e^{h_0(t-s)}\|\f\|_\infty,
\end{equation*}
for any $t>s \in I$, $\f \in C_b(\Rd;\Rm)$, where $h_0= \sup_{I\times \Rd}h$.
\end{prop}
\begin{proof}
Let $T>s$ and $J:=[s,T]$. Up to replacing $\lambda:=\lambda_J$ with a larger constant if needed, we can assume that
there exists a function $\varphi:=\varphi_J$ as in Hypothesis \ref{uni1}(ii) satisfying
$\sup_{\eta \in \partial B_1}\sup_{J\times\R^d}( \A_{\eta}\varphi-\lambda\varphi)<0$
with $\lambda>2h_0$. Now, for any $t\in J$, $x\in\R^d$ and $n\in\N$, we set
\[
v_n(t,x):=e^{-\lambda(t-s)}|\uu(t,x)|^2-e^{-(\lambda-2h_0)(t-s)}\|\f\|_{\infty}^2-\frac{\varphi(x)}{n}.
\]
where $\uu={\bf G}(\cdot,s)\f$.
Our aim consists in proving that $v_n\leq0$ in $[s,T]\times\R^d$ for any $n\in\N$. Indeed in this case letting $n\to +\infty$ and recalling that
$T$ has been arbitrarily fixed, we obtain $|\uu(t,\cdot)|^2\le e^{2h_0 (t-s)}\|\f\|_{\infty}^2$ in $\Rd$, for any $t \in [s,T]$
and the claim follows from the arbitrariness of $T>s$.

A straightforward computation shows that
\begin{align*}
D_tv_n(t,x)=e^{-\lambda(t-s)}&\left[(\A_0(t)+2h-\lambda)|\uu(t,\cdot)|^2-2V(D_1\uu(t,\cdot),\ldots,D_d\uu(t,\cdot),\uu(t,\cdot))
\right.\\
&\left.+ (\lambda-2h_0)e^{2h_0(t-s)}\|\f\|_{\infty}^2\right],
\end{align*}
in $(s,T]\times \Rd$, where $\A_0(t)={\rm div}(Q(t,\cdot)D_x)$ and
\begin{align*}
V(\cdot,\cdot,\xi^1,\ldots,\xi^d,\zeta) :=&
\sum_{i,j=1}^dq_{ij}\langle\xi^i,\xi^j\rangle-\sum_{j=1}^d\langle B_j\xi^j,\zeta\rangle-\langle (C-h)\zeta,\zeta\rangle
\end{align*}
for any $\xi^1,\ldots,\xi^d,\zeta\in\R^m$. Since $\lambda>2h_0$, we can estimate
\begin{align}
&D_tv_n(t,\cdot) -(\A_0(t)+2h-\lambda)v_n(t,\cdot)
-2(h-h_0)e^{-(\lambda-2h_0)(t-s)}\|{\bf f}\|_{\infty}^2\notag\\
<&\frac{1}{n}(\A_0(t)+2h-\lambda)\varphi
-2e^{-\lambda(t-s)}V(D_1\uu(t,\cdot),\ldots,D_d\uu(t,\cdot),\uu(t,\cdot)),
\label{davide-lecce}
\end{align}
in $\Rd$ for any $t\in (s,T]$. Since $\lim_{|x|\rightarrow+\infty}v_n(t,x)=-\infty$, uniformly with respect to $t\in [s,T]$, $v_n$
attains its maximum at some point $(t_0,x_0)\in [s,T]\times\R^d$. If $t_0=s$ the proof is complete since $v_n(s,\cdot)<0$.
If $t_0>s$, assume by contradiction that $v_n(t_0,x_0)>0$. In this case, since $\lambda-2h\ge 0$ in $I \times \Rd$,
the left-hand side of \eqref{davide-lecce} is strictly positive at $(t_0,x_0)$.

Thus, it suffices to prove that the right-hand side of \eqref{davide-lecce} is nonpositive at $(t_0,x_0)$
to get a contradiction and to conclude that $v_n\le 0$ in $[s,T]\times\R^d$.

Since $D_xv_n(t_0,x_0)=0$, it holds that
$\langle D_j\uu(t_0,x_0),\uu(t_0,x_0)\rangle=D_j\tilde\varphi(x_0)/(2n)$ for any $j=1,\ldots,d$,
 where $\tilde\varphi=e^{\lambda(t_0-s)}\varphi$. Thus it is enough to
show that the maximum of the function
\begin{align*}
F_{n,\zeta}(\xi^1,\ldots,\xi^d):=\frac{1}{n}(\A_0(t_0)+2h(t_0,\cdot)-\lambda)\tilde\varphi(x_0)-2V(t_0,x_0,\xi^1,\ldots,\xi^d,\zeta),
\end{align*}
in the set $\Sigma=\left\{(\xi^1,\ldots,\xi^d)\in\R^{md}: \langle \xi^j,\zeta\rangle=(2n)^{-1}D_j\tilde\varphi(x_0),\,j=1,\ldots,d\right\}$ is nonpositive.
Note that the function $(\xi^1,\ldots\xi^d)\mapsto V(t_0,x_0,\xi^1,\ldots,\xi^d,\zeta)$ tends to $+\infty$ as $\|(\xi^1,\ldots,\xi^d)\|\to +\infty$, for any $\zeta\in\R^m$.
Hence, $F_{n,\zeta}$ has a maximum in $\Sigma$ attained at some point
$(\xi^1_0,\ldots,\xi^d_0)$. Applying the Lagrange multipliers theorem, it can be proved that
\begin{align*}
\xi^j_0=&\frac{1}{2n}|\zeta|^{-2}\zeta D_j\tilde\varphi(x_0)+\frac{1}{2}\sum_{k=1}^d(Q^{-1})_{jk}(t_0,x_0)\left[
B_k(t_0,x_0)\zeta-|\zeta|^{-2}\langle B_k(t_0,x_0)\zeta,\zeta\rangle\zeta\right],
\end{align*}
for $j=1,\ldots,d$ and, consequently, that
\begin{align*}
V(t_0,x_0,\xi_0^1,\ldots,\xi_0^d)=&
\frac{1}{4n^2|\zeta|^2}|Q^{1/2}(t_0,x_0)D\tilde\varphi(x_0)|^2
-\langle (C(t_0,x_0)-h(t_0,x_0))\zeta,\zeta\rangle\\
&-\frac{1}{4}\sum_{i,k=1}^d(Q^{-1})_{ik}\langle B_i(t_0,x_0)\zeta, B_k(t_0,x_0)\zeta\rangle\\
&+\frac{1}{4|\zeta|^2}\sum_{i,k=1}^d(Q^{-1})_{ik}\langle  B_i(t_0,x_0)\zeta,\zeta\rangle \langle B_k(t_0,x_0)\zeta,\zeta\rangle
\\
&-\frac{1}{2n|\zeta|^2}\sum_{j=1}^dD_j\tilde\varphi(x_0)\langle B_j(t_0,x_0)\zeta,\zeta\rangle.
\end{align*}
It thus follows that
\begin{align*}
\max_{\Sigma} F_{n,\zeta}
=&\frac{1}{n}(\A_{\zeta/|\zeta|}(t_0)\tilde\varphi(x_0)-\lambda\tilde\varphi(x_0))
\\
&-\frac{1}{2n^2|\zeta|^2}|Q^{1/2}(t_0,x_0)D\tilde\varphi(x_0)|^2
-\frac{1}{2}|\zeta|^2{\mathcal K}(t_0,x_0,|\zeta|^{-1}\zeta)\le 0,
\end{align*}
and the proof is complete.
\end{proof}

\begin{coro}
Let assume that Hypotheses $\ref{hyp_base}$ hold true. Then, 
\begin{enumerate}[\rm (i)]
\item   if Hypotheses $\ref{uni2}$ are satisfied, then the classical solution $\uu$ of the problem \eqref{eq:cauchy_problem_system}
satisfies the estimate $\|\uu (t,\cdot)\|_\infty\le \|\f\|_\infty$, for any $t>s \in I$ and $\f \in C_b(\Rd;\Rm)$; 
\item   if Hypotheses $\ref{uni11}$ are satisfied, then the classical solution of the problem \eqref{prob_approx_dual} satisfies the estimate
$\|\vv(t, \cdot)\|_\infty\le e^{\kappa_0(t-s)}\|\f\|_\infty$, for any $t>s \in I$ and $\f \in C_b(\Rd;\Rm)$.
\end{enumerate}
\end{coro}

\end{document}